\documentclass[paper=a4, fontsize=12pt, abstract=on]{article}
\usepackage{titling}

\usepackage[headsepline]{scrlayer-scrpage}
\usepackage{geometry}

\usepackage[utf8]{inputenc}

\usepackage{lmodern}
\usepackage[T1]{fontenc}

\usepackage{graphicx}
\usepackage{amsmath}
\usepackage{amsthm}
\usepackage{amsfonts}
\usepackage{amssymb}
\usepackage{mathtools}
\usepackage{color}
\usepackage{dsfont}
\usepackage{mathrsfs}
\usepackage{color}

\allowdisplaybreaks

\newtheorem{thm}{Theorem}[section]
\newtheorem{lem}[thm]{Lemma}
\newtheorem{prop}[thm]{Proposition}

\theoremstyle{definition}
\newtheorem{defn}[thm]{Definition}

\theoremstyle{remark}
\newtheorem{rem}[thm]{Remark}

\numberwithin{equation}{section}

\newcommand{\abs}[1]{\left\vert#1\right\vert}

\newcommand{\HH}{\mathfrak{H}}

\newcommand{\rr}{\mathbb R}
\newcommand{\E}{\mathbb E}

\newcommand{\noi}{\noindent}

\newcommand{\R}{\mathbb R}

\newcommand{\Pp}{\mathbb P}

\newcommand{\N}{\mathbb N}

\newcommand{\Z}{\mathbb Z}

\newcommand{\Var}{{\operatorname{Var}}}

\def\eqfdd{\renewcommand{\arraystretch}{0.5}
\begin{array}[t]{c}
\stackrel{\rm fdd}{=} \\
\end{array}\renewcommand{\arraystretch}{1}}

\def\limd{\renewcommand{\arraystretch}{0.5}
\begin{array}[t]{c}
\stackrel{\rm d}{\longrightarrow} \\
\end{array}\renewcommand{\arraystretch}{1}}

\def\limfdd{\renewcommand{\arraystretch}{0.5}
\begin{array}[t]{c}
\stackrel{\rm fdd}{\longrightarrow} \\
\end{array}\renewcommand{\arraystretch}{1}}

\def\limfdd{\renewcommand{\arraystretch}{0.5}
\begin{array}[t]{c}
\stackrel{\rm fdd}{\longrightarrow} \\
\end{array}\renewcommand{\arraystretch}{1}}

\title{\bf How does tempering affect the local and global properties of fractional Brownian motion?}

\theoremstyle{plain}
\begin{document}

\author{Ehsan Azmoodeh
\thanks{Department of Mathematical Sciences,
University of Liverpool, Liverpool L69 7ZL, United Kingdom.
E-mail: \texttt{ehsan.azmoodeh@liverpool.ac.uk}} \qquad 
Yuliya Mishura
\thanks{Department of Probability Theory, Statistics and Actuarial Mathematics,
Taras Shevchenko National University of Kyiv,
	Volodymyrska, 64, Kyiv, 01601, Ukraine. E-mail: \texttt{myus@univ.kiev.ua}} \qquad 
Farzad Sabzikar
\thanks{Department of Statistics,
	Iowa State University, Ames, IA 50011, USA.
E-mail: \texttt{sabzikar@iastate.edu}}
}

\maketitle

\begin{abstract}
The present paper investigates the effects of tempering the power law kernel of moving average representation of a fractional Brownian motion (fBm) on some   local and global properties of this Gaussian stochastic process. Tempered fractional Brownian motion (TFBM) and tempered fractional Brownian motion of the second kind (TFBMII) are the processes that are considered in order to investigate the role of tempering. Tempering does not change the local properties of fBm including the sample paths and $p$-variation, but it has a strong impact on the Breuer-Major theorem, asymptotic behavior  of the 3rd and 4th cumulants of fBm and the optimal fourth moment theorem.
\end{abstract}

\vskip0.1cm

\tableofcontents

\section{Introduction}\label{sec-intr}

Fractional Brownian motion (fBm) is a Gaussian stochastic process whose increments, termed fractional Gaussian noise (fGn), can exhibit long range dependence in the sense that the power law spectral density of fGn blows up near the origin \cite{Beranbook, EmbrechtsMaejima, Pipirasbook, Samorodnitskybook}. A fBm has become popular in applications to science and engineering, since it yields a simple tractable model that captures the correlation structure seen in many natural systems \cite{Ascione, Harms, KolmogorovFBM, Mandelbrot1982, Meerschaertbook, Molz}.


Recently, two wide classes of continuous stochastic Gaussian processes which are called tempered fractional Brownian motion (TFBM) and tempered fractional Brownian motion of the second kind (TFBMII) were introduced in \cite{Meerschaertsabzikar} and \cite{SurgailisFarzadTFSMII}, respectively. Unlike the fBm, the TFBM and TFBMII can be defined for any value of the Hurst parameter $H>0$.
TFBM and TFBMII attracted the attention of researchers in various fields. It is known that bifurcation theory is very beneficial to survey
the qualitative or topological changes in the orbit structure of parameterized dynamical systems.
A new stochastic phenomenological bifurcation of the Langevin equation perturbed by TFBM was constructed in \cite{YangZengChen}. As a result, it was shown that the tempered fractional Ornstein-Uhlenbeck process, which is the solution of Langevin equation driven by a tempered fractional Brownian motion, exhibits very diverse and interesting bifurcation phenomena. The paper \cite{DengChenWang} discovered further properties of tempered fractional Brownian motion such as its ergodicity, and the derivation of the corresponding Fokker-Planck equation. Furthermore, they
argued carefully that the asymptotic form of the mean squared displacement of the
tempered fractional Langevin equation transits from $t^2$ (ballistic diffusion for short time) to $t^{2-2H}$ , and then to ${t^2}$ (again ballistic diffusion for long time). The arbitrage opportunities for the Black-Scholes model driven by TFBM was
investigated in \cite{ZhangXiao}.  In \cite{SWP}, the authors developed the asymptotic theory for the ordinary least squares estimator of the unknown coefficient of an autoregressive model of order one when the additive error follows a discrete tempered linear process. Consequently, they showed that the limiting results involves TFBMII under some conditions. Finally, \cite{Lim} introduced some extensions of TFBM including 
Mixed TFBM and tempered multifractional Brownian motion and studied essential properties of these stochastic processes.

TFBM and TFBMII can also be useful stochastic models for situations where the data follows fBm at some intermediate scale, but then deviates from fBm at longer scale.  For example, wind speed measurements typically resemble long range dependent fBm over a range of frequencies, but deviate significantly at very low frequencies (corresponding to very long spatial scales). Since the spectral density of the increments of tempered fractional processes follows the same pattern, they can provide a useful model for such data.  Recently, \cite{BDS} showed that TFBM can display the Davenport spectrum which is a modification of the Kolomogorov $-5/3$ power law spectrum for the inertial range in turbulence. The same paper, also used the wavelets to study the statistical inference including the parameter estimation for TFBM and also hypothesis test for fBm vs TFBM.

The aim of this paper is to study more deeply properties of tempered fractional processes. In fact, our task is to investigate the effects of tempering on some local and global properties of fBm.

In Section \ref{local}, we study the consequences of the tempering on the behavior of the variance, sample paths, and $p$-variation of a fBm. Proposition \ref{lem:usefulrelations 1} gives the
asymptotic properties of the variance of the tempered fractional processes for large time $t$.
Unlike the fBm,  the variance of the TFBM converges to a finite constant  when the time converges to infinity. However, the variance of TFBMII is proportional to $ct$ for large $t$. Therefore, the tempering makes TFBM and TFBMII as stochastically bounded and stochastically unbounded for large $t$, respectively. On the other hand, tempering does not change the sample paths behavior of fBm, see Lemma \ref{lem:TFBM and TFBMII bounds}. Consequently, we prove the existence of local times for TFBM and TFBMII. We will also show that these tempered processes are locally nondeterministic on every compact interval, see Proposition \ref{prop:tfbmIIlocaltime} and Theorem \ref{prop:TFBMIILND}. In Subsection \ref{sec3}, we show that the tempering keeps the same $\frac{1}{H}$-variation for the TFBM and TFBMII. Tempered fractional processes demonstrate the  so called   semi-long range dependence or semi-long memory. Their increments' covariance function resembles long range dependence for a number of lags, depending on the tempering parameter, but eventually decays exponentially fast. The spectral density of tempered fractional Gaussian noise, TFGN, is zero at the origin so that TFBM is anti-persistence process, while the spectral density of TFGNII remains bounded at very low frequencies, see \cite{Meerschaertsabzikar,SurgailisFarzadTFSMII}. The semi-long range dependence property and behavior of the spectral density of tempered fractional processes motivate us to study Breuer-Major theorem for TFGNs in Section \ref{sec:BM}.  The section begins with Lemma \ref{lem:Positive-Negative-Correlation} revealing an interesting switching feature on correlation structure of the tempered fractional Gaussian noise of the first kind. Then, we continue studying the effect of tempering on the popular Breuer--Major Theorem and its modern ramifications in the realm of Malliaivn--Stein method.  Furthermore,
we investigate the asymptotic behavior  of the third and fourth cumulants of
tempered fractional Gaussian processes. The tempering parameter $\lambda$ manifests its role in the optimal fourth
moment as well, see Remark \ref{rem:optimalrate}.\\

In what follows, $C, C_i$ for $i=1,2, \ldots$ denote generic constants
which may be different at different locations.
We write $ \limd $, $ \limfdd $ and  $   \eqfdd $
 for weak convergence in distribution, weak convergence and equality
 of finite-dimensional
distributions, respectively. Also, denote $\R_+ := (0, \infty),
(x)_\pm := \max (\pm x, 0),  x \in \R,  \, \int := \int_\R$.
For two non-negative sequences $a_n$ and $b_n$, we write $a_n\asymp b_n$ to indicate that  $0<\liminf_{n\to\infty}\frac{a_n}{b_n}\leq \limsup_{n\to\infty}\frac{a_n}{b_n}<\infty$.  The relation $f\sim g$ means that $\lim_{z\to\infty}\frac{f(z)}{g(z)}=1$. Some other notations are given in Appendices A and B.

\section{The effects of tempering on the local properties of fBm}\label{local}

\subsection{Asymptotic behavior of the variations of tempered fractional processes}\label{sec2}

Let $\{B(t)\}_{ t\in \rr}$ be a real-valued Brownian motion on the real line, i.e., a zero mean Gaussian process with stationary independent increments and    variance $|t|$ for all $t\in\rr$.  Define an independently scattered Gaussian random measure $B(dx)$ with control measure $m(dx)= dx$ by setting $B((a,b])=B(b)-B(a)$ for any real numbers $a<b$, and then extending to all Borel sets with finite Lebesgue measure.  Then the Wiener integral  $I(f):=\int f(x)B(dx)$ is defined for all functions $f:\rr\to\rr$ such that $\int f(x)^2dx<\infty$, as Gaussian random variable  with  zero mean and covariance $\E[I(f)I(g)]= \int f(x)g(x)dx$. Moreover, well-known Mandelbrot-van-Ness representation of two-sided normalized fractional Brownian motion (fBm) with Hurst index $H\in (0,1)\setminus\left\{1/2\right\}$ has a form
\begin{equation*} B_H(t)=C_H\int\left((t-x)_{+}^{H-\frac{1}{2}}-(-x)_{+}^{H-\frac{1}{2}}\right)B(dx),
\end{equation*}
where  $C_H=\frac{(\Gamma(2H+1)\sin(\pi H))^{1/2}}{\Gamma(H+1/2)}$, see \cite{Mishura}. Now our goal is to modify this representation as follows.
On the one hand, it is possible simply to moderate the integrand by exponent, ignore the normalizing constant  and give the following definition.
\begin{defn}\label{defTFBM}
Given an independently scattered Gaussian random measure $B(dx)$ on $\rr$ with control measure $dx$, for any $H>0$ and $\lambda>0$, the stochastic process $B^I_{H,\lambda}= \{ B^I_{H,\lambda}(t) \}_{t\in \rr}$ defined by the Wiener integral
\begin{equation}\label{eq:TFBMdef}
B^{I}_{H,\lambda}(t):={\int \left[e^{-\lambda(t-x)_{+}}(t-x)_{+}^{H-\frac{1}{2}}-e^{-\lambda(-x)_{+}}(-x)_{+}^{H-\frac{1}{2}}\right]\ B(dx)},
\end{equation}
where $0^0=0$, is called the tempered fractional Brownian motion (TFBM).
\end{defn}

It is easy to check that the function
\begin{equation*}
g^{I}_{H,\lambda,t}(x):=e^{-\lambda(t-x)_{+}}(t-x)_{+}^{H-\frac{1}{2}}-e^{-\lambda(-x)_{+}}(-x)_{+}^{H-\frac{1}{2}}
\end{equation*}
is square integrable over the entire real line for any $H>0$, so that TFBM is well-defined. Note that it is defined for $H=1/2$
as well, in contrast to the Mandelbrot-van-Ness representation, and  equals
$$B^{I}_{1/2,\lambda}(t)=e^{-\lambda t}\int_{-\infty}^t e^{ \lambda x}\ B(dx)-\int_{-\infty}^0e^{ \lambda x}\ B(dx).$$
However, in what follows, we shall consider mostly $H\neq 1/2$.

On the other hand, for $H\neq 1/2$,  the kernel $(t-x)_{+}^{H-\frac{1}{2}}-(-x)_{+}^{H-\frac{1}{2}}$ can be represented as
$$(t-x)_{+}^{H-\frac{1}{2}}-(-x)_{+}^{H-\frac{1}{2}}=(H-1/2) \int_0^t(s-x)_{+}^{H-\frac{3}{2}}ds.$$  Moderating respectively the integrand by the same exponent, integrating by parts and ignoring normalizing constant,
 we get   another tempered stochastic process.
\begin{defn}\label{defTFBMII}
Given an independently scattered Gaussian random measure $B(dx)$ on $\rr$ with control measure $dx$, for any $H>0$ and $\lambda>0$, the stochastic process $B^{I\!I}_{H,\lambda}= \{ B^{I\!I}_{H,\lambda}(t) \}_{t\in \rr}$ defined by the Wiener integral
\begin{equation}\label{eq:TFBMIIdef}
B^{I\!I}_{H,\lambda}(t):= \int  g^{I\!I}_{H,\lambda,t}(x) B(dx),
\end{equation}
where
\begin{equation}\label{hdef0}
\begin{split}
g^{I\!I}_{H,\lambda,t}(x)&:=(t-x)_+^{H - \frac{1}{2}}e^{-\lambda (t-x)_+} - (-x)_+^{H - \frac{1}{2}} e^{-\lambda (-x)_+}\\
&  \quad + \lambda \int_{0}^{t} (s-x)_{+}^{H-\frac{1}{2}}e^{-\lambda(s-x)_{+}}\ ds, \quad x \in \R.
\end{split}
\end{equation}
is called the tempered fractional Brownian motion of the second kind (TFBMII).
\end{defn}

We also  note  that
TFBM \eqref{eq:TFBMdef} and TFBMII \eqref{eq:TFBMIIdef} are Gaussian stochastic processes with stationary increments, having the following   scaling property: for any scaling factor $c>0$
\begin{equation}\label{eq:scalingproperty}
\left\{X_{H,\lambda}(ct)\right\}_{t \in \rr}{\triangleq}\left\{c^{H}X_{H,c\lambda}(t)\right\}_{t \in \rr}
\end{equation}
where $X_{H,\lambda}$ could be $B^{I}_{H,\lambda}$ or $B^{I\!I}_{H,\lambda}$   (see  \cite[Proposition 2.2]{Meerschaertsabzikar} and \cite[Proposition 2.9]{SurgailisFarzadTFSMII}). Using the scaling property \eqref{eq:scalingproperty} and the fact that
 $X_{H,\lambda}(|t|)$ has the same distribution as $|t|^{H} X_{H,\lambda|t|}(1)$, it is easy to see that $\mathbb{E}[(X_{H,\lambda}(|t|))^{2}]=|t|^{2H}\mathbb{E}[(X_{H,\lambda | t|}(1))^{2}]=: |t|^{2H}C^{2}_{t}$. Next, we recall
an explicit representation for $C^{2}_{t}$. We refer the reader to \cite{Meerschaertsabzikar,SurgailisFarzadTFSMII} for the details.

\begin{lem}\label{lem:usefulrelations}
\noi(a) Let $X_{H,\lambda} = B^{I}_{H,\lambda}$. Then  the function $C_{t}^{2}= (C^{I}_{t})^{2}=\mathbb{E}[(B^{I}_{H,\lambda |t|}(1))^{2}]$ has the expression
\begin{equation}\label{eq:CtTFBMDef}
(C^{I}_{t})^{2}=\frac{2\Gamma(2H)}{(2\lambda |t|)^{2H}}-\frac{2\Gamma(H+\frac 12)}{\sqrt{\pi}}\frac{1}{(2\lambda |t|)^{H}}K_{H}(\lambda |t|),
\end{equation}
where $t \neq 0$ and  $K_{\nu}(z)$ is the modified Bessel function of the second kind (see Appendix A for the definition of $K_{\nu}(z)$).

\noi(b) Let $X_{H,\lambda} = B^{I\!I}_{H,\lambda}$. Then  the function $C_{t}^{2}= (C^{I\!I}_{t})^{2}=\mathbb{E}[(B^{I\!I}_{H,\lambda |t|}(1))^{2}]$ has the expression
\begin{equation}\label{eq:CtTFBIIMDef}
\begin{split}
(C^{I\!I}_{t})^{2}&=\frac{(1-2H) \Gamma(H+\frac{1}{2})  \Gamma(H)(\lambda t)^{-2H} }{\sqrt{\pi}}
\Big[1-{_2F_3}{ \Big(\{1,-1/2\}, \{1-H,1/2,1\},  \lambda^2 t^2/4\Big)}\Big]\\
&\quad + \frac{\Gamma (1-H)  {\Gamma(H+\frac{1}{2})}}{ \sqrt{\pi} H 2^{2H}}\,{_2F_3} \Big(\{1,H- 1/2\},
\{1,H+1,H+ 1/2\}, \lambda^2 t^2/4\Big),
\end{split}
\end{equation}
where ${_2F_3}$ is the generalized hypergeometric function (see Appendix A for the definition of ${_2F_3}$).
\end{lem}

\begin{prop}\label{lem:usefulrelations 1}
\begin{itemize}
\item[$(a)$] The {\rm TFBM} \eqref{eq:TFBMdef} with parameters $H>0$ and $\lambda>0$ satisfies
\begin{equation}\label{eq:varasy}
  \lim _{t\to  +\infty}\E[B^I_{H,\lambda}(t)]^2=\frac{2\Gamma{(2H)}}{(2\lambda)^{2H}}.
  \end{equation}
\item[$(b)$] The {\rm TFBMII} \eqref{eq:TFBMIIdef} with parameters $H>0$ and $\lambda>0$ satisfies
\begin{equation}\label{eq:varasyII}
  \lim _{t\to +\infty}\E\Big[\frac{B^{I\!I}_{H,\lambda}(t)}{\sqrt{t}}\Big]^2=\lambda^{1-2H}\Gamma^{2}\left(H+\frac{1}{2}\right).
  \end{equation}
  \end{itemize}
  \end{prop}

\begin{proof}
The proof of part $(a)$ follows from the fact that
\begin{equation*} 
K_{\nu}(z)\sim\sqrt{\frac{\pi}{2}}\frac{e^{-z}}{\sqrt{z}}
\end{equation*}
as $z\to\infty$.
Part $(b)$:  Obviously, for any $t\ge 0$,
\begin{equation*}\begin{gathered}\E\left( B^{I\!I}_{H,\lambda}(t) \right)^2=\int_{-\infty}^0\left((t-x)^{H-\frac{1}{2}}e^{-\lambda (t-x)} - (-x)^{H - \frac{1}{2}} e^{-\lambda (-x)}
 + \lambda \int_{0}^{t} (s-x)^{H-\frac{1}{2}}e^{-\lambda(s-x)}\ ds\right)^2dx\\+
 \int_{0}^t\left((t-x)^{H-\frac{1}{2}}e^{-\lambda (t-x)}
 + \lambda \int_{x}^{t} (s-x)^{H-\frac{1}{2}}e^{-\lambda(s-x)}\ ds\right)^2dx\\=
 \int_{0}^\infty\left((t+x)^{H-\frac{1}{2}}e^{-\lambda (t+x)} - x^{H - \frac{1}{2}} e^{-\lambda x}
 + \lambda \int_{0}^{t} (s+x)^{H-\frac{1}{2}}e^{-\lambda(s+x)}\ ds\right)^2dx\\+
 \int_{0}^t\left(u^{H-\frac{1}{2}}e^{-\lambda u}
 + \lambda \int_{0}^{u} v^{H-\frac{1}{2}}e^{-\lambda v}\ dv\right)^2du
 =\left(H-1/2\right)^2\int_{0}^{\infty}\left(\int_{0}^{t}(s+x)^{H-3/2}e^{-\lambda(s+x)}ds\right)^2dx\\
 +\int_{0}^{t}u^{2H-1}e^{-2\lambda u}du+2\lambda \int_{0}^{t}u^{H-1/2}e^{-\lambda u}\int_{0}^{u}v^{H-1/2}e^{-\lambda v}dvdu\\
  +\lambda^2\int_{0}^{t}\left(\int_{0}^{u}v^{H-1/2}e^{-\lambda v}dv\right)^2du=:\sum_{k=1}^4 I_k(t).
\end{gathered}\end{equation*}
It is easy to see that  $$\lim_{t\rightarrow\infty}I_1(t)=I_1(\infty)
=\left(H-1/2\right)^2\int_{0}^{\infty}\left(\int_{0}^{\infty}(s+x)^{H-3/2}e^{-\lambda(s+x)}ds\right)^2dx,$$
and this integral is finite, see Lemma \ref{lem:propo24} in Appendix A. Further,
$$\lim_{t\rightarrow\infty}I_2(t)=I_2(\infty)=(2\lambda)^{-2H}\Gamma(2H),$$
and, according to L'H\^{o}pital's rule,
$$\lim_{t\rightarrow\infty}t^{-1}I_3(t)= 2\lambda \lim_{t\rightarrow\infty}t^{H-1/2}e^{-\lambda t}\int_{0}^{t}v^{H-1/2}e^{-\lambda v}dv=0.$$
Finally, again according to L'H\^{o}pital's rule,
$$\lim_{t\rightarrow\infty}t^{-1}I_4(t)=\lambda^2 \left(\int_{0}^{\infty}v^{H-1/2}e^{-\lambda v}dv\right)^2
=\lambda^{1-2H}\Gamma^2(H+1/2),$$
and the proof follows.
\end{proof}

\begin{rem}
Since TFBM is a Gaussian stochastic process with zero mean, it follows from \eqref{eq:varasy} that $B^{I}_{H,\lambda}(t)$ converges in law to a normal random variable with  zero mean and variance $2\Gamma(2H)(2\lambda)^{-2H}$ as $t\to\infty$, unlike fBm, whose variance diverges to infinity. In contrast, relation \eqref{eq:varasyII} shows that TFBMII is stochastically unbounded as $t\to\infty$.
\end{rem}

The following proposition gives the covariance structure of TFBM and TFBMII (see \cite{Meerschaertsabzikar, SurgailisFarzadTFSMII} for more details).

\begin{prop}\label{covTFBMTFBMII}
\begin{itemize}
\item[$(a)$] {\rm TFBM} \eqref{eq:TFBMdef} has the covariance function
\begin{equation*}
{\rm Cov}\left[B^{I}_{H,\lambda}(t),B^{I}_{H,\lambda}(s)\right]=\frac {1}{2}\left[C_t^{2}\left|t\right|^{2H}
+C_s^{2}\left|s\right|^{2H}
-C_{t-s}^{2}\left|t-s\right|^{2H}\right]
\end{equation*}
for any $s,t\in\rr$, where $C^{2}_{t}= (C^{I}_{t})^{2}$ is given by \eqref{eq:CtTFBMDef}.

\item[$(b)$] {\rm TFBMII} \eqref{eq:TFBMIIdef} has the covariance function
\begin{equation*}
{\rm Cov}\left[B^{I\!I}_{H,\lambda}(t),B^{I\!I}_{H,\lambda}(s)\right]=\frac {1}{2}\left[C_t^{2}\left|t\right|^{2H}
+C_s^{2}\left|s\right|^{2H}
-C_{t-s}^{2}\left|t-s\right|^{2H}\right]
\end{equation*}
for any $s,t\in\rr$, where $C^{2}_{t}= (C^{I\!I}_{t})^{2}$ is given by \eqref{eq:CtTFBIIMDef}.
\end{itemize}

\end{prop}

\subsection{Sample paths properties and local times of tempered fractional processes}
Now we prove the existence of local times for tempered fractional processes.
To start with, we prove the following result that  will be used in    this section.
\begin{thm}\label{lem:TFBM and TFBMII bounds}
  Let $X$ stand for  be a TFBM $B^I_{H,\lambda}$ from \eqref{eq:TFBMdef} or for a TFBMII  $B^{I\!I}_{H,\lambda}$ from  \eqref{eq:TFBMIIdef} both with $0<H<1$ and $\lambda>0$. Then there exist positive constants $C_1$ and $C_2$ such that
\begin{equation}\label{two-sided1}
C_1 \left|t-s\right|^{2H}\leq \mathbb{E}[|X(t)-X(s)|^{2}]\leq C_2 \left|t-s\right|^{2H}
\end{equation}
 for any $s,t\in [0,1]$.

\end{thm}
\begin{rem} \label{rem2}
(i) Inequalities mean that both processes, TFBM  and TFBMII,  are quasi-h\'{e}lices, according to geometric terminology of J.-P. Kahane, see \cite{Kahane}.\\
(ii) Theorem \ref{lem:TFBM and TFBMII bounds} holds for any fixed interval $[0,T]$ with constants $C_i$   depending on $T$.
\end{rem}
\begin{proof}
The proof for TFBM is similar to that of Lemma 4.2 in \cite{Meerschaertsabzikar2} and hence can be omitted.
To prove \eqref{two-sided1} for TFBMII, we use its  moving average representation   \eqref{eq:TFBMIIdef} to write
\begin{equation}\label{eq:incrementsTFBMII2}
\begin{split}
\mathbb{E}\Big| B^{I\!I}_{H,\lambda}(t) - B^{I\!I}_{H,\lambda}(s) \Big|^2 &=   \int_{-\infty}^{s}
\Big[(t-x)^{H-\frac{1}{2}}e^{-\lambda(t-x)} - (s-x)^{H-\frac{1}{2}}e^{-\lambda(s-x)}\\
& \quad + \lambda \int_s^t (w-x)^{H-\frac{1}{2}}e^{-\lambda(w-x)}\ dw   \Big]^2 dx   \\
&\quad +  \int_s^t \Big[ (t-x)^{H-\frac{1}{2}}e^{-\lambda(t-x)} + \lambda\int_x^t (w-x)^{H-\frac{1}{2}}e^{-\lambda(w-x)}\ dw\Big]^2 dx \\
&= I_1 + I_2.
\end{split}
\end{equation}
Now, let $\frac{1}{2}< H <1$.  Then $I_{1}= I_{1, H>\frac{1}{2}}$ can be written as
\begin{equation}\label{I1TFBMIIHGRQ12}
I_{1, H>\frac{1}{2}} = {\left(H-\frac{1}{2}\right)^2 }\int_{-\infty}^s \Big[ \int_{s}^t  (w-x)^{H-\frac{3}{2}} e^{-\lambda(w-x)} \ dw \Big]^2 dx.
\end{equation}
Obviously,
\begin{equation}\label{I11TFBMIIHGRQ12}
\begin{split}
I_{ 1, H>\frac{1}{2} } &\leq \left(H-\frac{1}{2}\right)^2\int_{-\infty}^{s}  \Big( (t-x)^{H-\frac{1}{2}} - (s-x)^{H-\frac{1}{2}}  \Big)^{2}\ dx\\
& = \left(H-\frac{1}{2}\right)^2\int_{0}^{\infty} \Big( ( h+u)^{H-\frac{1}{2}} - u^{H-\frac{1}{2}}  \Big)^{2}\ du \quad (h:=t-s)\\
 &=\left(H-\frac{1}{2}\right)^2h^{2H} \int_{0}^{\infty} \Big( (1+ u)^{H-\frac{1}{2}} - u^{H-\frac{1}{2}}  \Big)^{2}\ du\\
&=C h^{2H} = C(t-s)^{2H},
\end{split}
\end{equation}
where we used the fact that $\int_{0}^{\infty} \Big( (1+ u)^{H-\frac{1}{2}} - u^{H-\frac{1}{2}}  \Big)^{2}\ du$ is finite, see, e.g.,  \cite[Theorem 1.3.1]{Mishura}.  Now, let's move on to consider  $I_2=I_{ 2, H>\frac{1}{2} }$,  and  get that
\begin{equation}\label{I2TFBMIIHGRQ12}
\begin{split}
I_{ 2, H>\frac{1}{2} } &= \int_s^t \Big[ (t-x)^{H-\frac{1}{2}}e^{-\lambda(t-x)} + \lambda\int_x^t (w-x)^{H-\frac{1}{2}}e^{-\lambda(w-x)}\ dw\Big]^2 dx\\
&= \left(H-\frac{1}{2}\right)^2\int_s^t \Big[ \int_x^t (w-x)^{H-\frac{3}{2}}e^{-\lambda(w-x)}\ dw\Big]^{2}\ dx   \leq C (t-s)^{2H}.
\end{split}
\end{equation}
We conclude  from \eqref{eq:incrementsTFBMII2}--\eqref{I2TFBMIIHGRQ12} that
\begin{equation*}
\mathbb{E}\left| B^{I\!I}_{H,\lambda}(t) - B^{I\!I}_{H,\lambda}(s) \right|^2 \leq C|t-s|^{2H}
\end{equation*}
provided $\frac{1}{2}< H < 1$. As the next step, we find an upper bound for the second moments of the increments of TFBMII for  $0<H<\frac{1}{2}$. Recall from \eqref{eq:incrementsTFBMII2} that in this case too we have
\begin{equation*}\label{eq:incrementsTFBMIIsecond}
\mathbb{E}\left| B^{I\!I}_{H,\lambda}(t) - B^{I\!I}_{H,\lambda}(s) \right|^2  = I_1 + I_2.
\end{equation*}
For $0<H<\frac{1}{2}$, we start with $I_2$ and write
\begin{equation*}
\begin{split}
I_{2}&= I_{2,H<\frac{1}{2}} = \int_s^t \left[ (t-x)^{H-\frac{1}{2}}e^{-\lambda(t-x)} + \lambda\int_x^t (w-x)^{H-\frac{1}{2}}e^{-\lambda(w-x)}\ dw\right]^2 dx\\
& = \int_s^t (t-x)^{2H-1} e^{-2\lambda(t-x)}\  dx
 + \lambda^2 \int_s^t \Big[ \int_{x}^{t} (w-x)^{H-\frac{1}{2}}e^{-\lambda(w-x)} \ dw \Big]^{2}\  dx\\
& \quad + 2\lambda \int_s^t (t-x)^{H-\frac{1}{2}}e^{-\lambda(t-x)} \int_x^t (w-x)^{H-\frac{1}{2}}e^{-\lambda(w-x)}\ dw dx\\
&\le   \int_s^t (t-x)^{2H-1}   dx
 + \lambda^2 \int_s^t \Big[ \int_{x}^{t} (w-x)^{H-\frac{1}{2}}  \ dw \Big]^{2}\  dx\\
&\quad + 2\lambda \int_s^t (t-x)^{H-\frac{1}{2}}  \int_x^t (w-x)^{H-\frac{1}{2}} \ dw dx\leq C|t-s|^{2H}.
\end{split}
\end{equation*}
Next, we consider $I_{1}=I_{1, H<\frac{1}{2}}$ and decompose it into three terms as follows:
\begin{equation*}
\begin{split}
I_1&= I_{1, H<\frac{1}{2}} = \int_{-\infty}^{s}
\left( (t-x)^{H-\frac{1}{2}}e^{-\lambda(t-x)} - (s-x)^{H-\frac{1}{2}}e^{-\lambda(s-x)}\right)^2\ dx\\
& \quad + \lambda^2 \int_{-\infty}^{s}\left( \int_s^t (w-x)^{H-\frac{1}{2}}e^{-\lambda(w-x)}\ dw   \right)^2 dx   \\
& \quad + 2\lambda \int_{-\infty}^{s}
\Bigg( (t-x)^{H-\frac{1}{2}}e^{-\lambda(t-x)} - (s-x)^{H-\frac{1}{2}}e^{-\lambda(s-x)}\Bigg)
\Bigg( \int_s^t (w-x)^{H-\frac{1}{2}}e^{-\lambda(w-x)}\ dw   \Bigg) \ dx\\
&=: I_{11, H<\frac{1}{2}} + I_{12, H<\frac{1}{2}} + I_{13, H<\frac{1}{2}}.
\end{split}
\end{equation*}
According to \cite[Lemma 4.2]{Meerschaertsabzikar2},
\begin{equation}\label{equat:11}
I_{11, H<\frac{1}{2}} \leq C |t-s|^{2H}
\end{equation}
provided $s,t\in (0,1)$. Let as before, $h=t-s$. Taking into account that  for any $0\le y\le z$ and $H<1/2$ we have $z^{H+\frac{1}{2}}-y^{H+\frac{1}{2}}\le (z-y)^{H+\frac{1}{2}} $, the term $I_{12, H<\frac{1}{2}}$ can be rewritten as
\begin{equation}\label{equat:12}
\begin{split}
I_{12, H<\frac{1}{2}} &= \lambda^2 \int_{-\infty}^{s}\Bigg( \int_s^t (w-x)^{H-\frac{1}{2}}e^{-\lambda(w-x)}\ dw   \Bigg)^2 dx  \\
&\leq \lambda^2 \Big(H+\frac{1}{2}\Big)^{-2} \int_{-\infty}^{s} e^{-2\lambda(s-x)} \Bigg( (t-x)^{H+\frac{1}{2}} - (s-x)^{H+\frac{1}{2}}  \Bigg)^2 dx \\
&= \lambda^2 \Big(H+\frac{1}{2}\Big)^{-2} \int_{0}^{\infty} e^{-2\lambda u} \Bigg( (u+h)^{H+\frac{1}{2}} - u^{H+\frac{1}{2}}  \Bigg)^2 du \\
 &\le \lambda^2 \Big(H+\frac{1}{2}\Big)^{-2} h^{2H+1} \int_{0}^{\infty} e^{-2\lambda u} du
 \leq Ch^{2H+1}\leq Ch^{2H},
\end{split}
\end{equation}
where the last inequality is due to $0<h=t-s<1$. Next, for $I_{13, H<\frac{1}{2}}$  we have
\begin{equation}\label{I13forHleq12}
\begin{split}
I_{13, H<\frac{1}{2}}&= 2\lambda \int_{-\infty}^{s}
\Bigg( (t-x)^{H-\frac{1}{2}}e^{-\lambda(t-x)} - (s-x)^{H-\frac{1}{2}}e^{-\lambda(s-x)}\Bigg)\\
& \quad \times\Bigg( \int_s^t (w-x)^{H-\frac{1}{2}}e^{-\lambda(w-x)}\ dw   \Bigg) \ dx \\
&\leq 2\lambda \Big(H+\frac{1}{2}\Big)^{-1} \int_{-\infty}^{s}
(t-x)^{H-\frac{1}{2}}e^{-\lambda(t-x)}\left[(t-x)^{H+\frac{1}{2}} - (s-x)^{H+\frac{1}{2}} \right]\ dx\\
& \quad +2\lambda \Big(H+\frac{1}{2}\Big)^{-1} \int_{-\infty}^{s}
(s-x)^{H-\frac{1}{2}}e^{-\lambda(s-x)}\left[(t-x)^{H+\frac{1}{2}} - (s-x)^{H+\frac{1}{2}} \right]\ dx\\
&=: I_{131, H<\frac{1}{2}} + I_{132, H<\frac{1}{2}}.
\end{split}
\end{equation}
Note that the function $(1+ u)^{H-\frac{1}{2}}\left[(1+u)^{H+\frac{1}{2}} - u^{H+\frac{1}{2}} \right]$ is bounded on $[0,\infty)$, and continue with $I_{131, H<\frac{1}{2}}$ changing a variable $s-x=hu$:
\begin{equation}\label{I131}
\begin{split}
I_{131, H<\frac{1}{2}} &= 2\lambda \Big(H+\frac{1}{2}\Big)^{-1}   \int_{-\infty}^{s}
(s+h-x)^{H-\frac{1}{2}}e^{-\lambda(s+h-x)}\left[(s+h-x)^{H+\frac{1}{2}} - (s-x)^{H+\frac{1}{2}} \right]\ dx\\
 &=2\lambda \Big(H+\frac{1}{2}\Big)^{-1} h^{2H+1} e^{-\lambda h} \int_{0}^{\infty}
(1+ u)^{H-\frac{1}{2}} e^{-\lambda hu} \left[(1+u)^{H+\frac{1}{2}} - u^{H+\frac{1}{2}} \right]\ du\\
&\leq Ch^{2H+1}\int_{0}^{\infty}e^{-\lambda hu}du \leq Ch^{2H}.
\end{split}
\end{equation}
 For $I_{132, H<\frac{1}{2}}$ corresponding function $u^{H-\frac{1}{2}}\Bigg[(1+u)^{H+\frac{1}{2}} - u^{H+\frac{1}{2}} \Bigg]$ is not bounded at zero, therefore we change a bit the transformations:
\begin{equation}\label{I132}
\begin{split}
I_{132} &= 2\lambda \Big(H+\frac{1}{2}\Big)^{-1} \int_{-\infty}^{s}
(s-x)^{H-\frac{1}{2}}e^{-\lambda(s-x)}\Bigg[(s+h-x)^{H+\frac{1}{2}} - (s-x)^{H+\frac{1}{2}} \Bigg]\ dx\\
&= 2\lambda h^{2H+1} \Big(H+\frac{1}{2}\Big)^{-1}  \int_{0}^{\infty}
u^{H-\frac{1}{2}} e^{-\lambda hu} \Bigg[ (1+ u)^{H+\frac{1}{2}} - u^{H+\frac{1}{2}} \Bigg]\ du \\
&\leq C h^{2H+1}\left(\int_{0}^{1}+\int_{1}^{\infty} \right)\leq C \left(h^{2H+1}+h^{2H}\right)\leq Ch^{2H},
\end{split}
\end{equation}
where the last inequality is due to $0<h=t-s<1$.
From \eqref{I13forHleq12}-- \eqref{I132} we have
\begin{equation*}
I_{13, H<\frac{1}{2}}\leq C|t-s|^{2H},
\end{equation*}
and together with \eqref{equat:11} and \eqref{equat:12} it gives us the upper bound
\begin{equation*}
I_{1}= I_{11, H<\frac{1}{2}}+ I_{12, H<\frac{1}{2}} + I_{13, H<\frac{1}{2}} \leq C|t-s|^{2H}.
\end{equation*}
Therefore,
\begin{equation*}
\mathbb{E}\Big| B^{I\!I}_{H,\lambda}(t) - B^{I\!I}_{H,\lambda}(s) \Big|^2 \leq C|t-s|^{2H}
\end{equation*}
provided $0< H < \frac{1}{2}$. So, we established  that the right-hand side of \eqref{two-sided1} holds for any
for any $0<H<1$. Next, for $0<H<1$, let us prove that
\begin{equation*}
\mathbb{E}\Big| B^{I\!I}_{H,\lambda}(t) - B^{I\!I}_{H,\lambda}(s) \Big|^2 \geq C|t-s|^{2H}.
\end{equation*}
 In order to obtain the required lower bound, it suffices to note that formula  \eqref{eq:incrementsTFBMII2}
allows us to write for $s,t\in[0,1]$:
\begin{equation*}
\begin{split}
\mathbb{E}\Big| B^{I\!I}_{H,\lambda}(t) - B^{I\!I}_{H,\lambda}(s) \Big|^2 &\geq \int_s^t \Big[ (t-x)^{H-\frac{1}{2}}e^{-\lambda(t-x)}\Big]^2 dx\\ &\geq e^{-2\lambda(t-s)}\int_s^t   (t-x)^{2H-1}  dx
 \geq   C (t-s)^{2H},
\end{split}
\end{equation*}
and   the  proof   is complete.
\end{proof}

Next, we prove the existence of local times for TFBM and TFBMII. We will also show that these tempered fractional processes are locally nondeterministic on any open interval. Suppose $X=\{X(t)\}_{t\geq 0}$ is a real-valued separable random process with Borel sample functions. The random Borel measure
\begin{equation*}
\mu_{B}(A)=\int_{  B}I\{X(s)\in A\}\ ds
\end{equation*}
defined for Borel sets $A\subseteq {\mathbb R},\;B\subseteq {\mathbb R}^{+}$ is called the occupation measure of $X$ on $B$.  If $\mu_{B}$ is absolutely continuous with respect to Lebesgue measure on $\mathbb{R}^{+}$, then the Radon-Nikodym derivative of $\mu_{B}$ with respect to Lebesgue measure is called the local time of $X$ on $B$, denoted by $L(B,x)$.  See Boufoussi et al.\ \cite{BoufoussiDozziGuerbaz} for more detail.  For brevity, we denote   $L(t,x):=L([0,t],x)$.

\begin{prop}\label{prop:tfbmIIlocaltime}
Let $X$ be either TFBM \eqref{eq:TFBMdef} or TFBMII \eqref{eq:TFBMIIdef}. Then  for $0< H < 1$ and $\lambda>0$, $X$ has a local time $L(t,x)$ that is continuous in $t$ for a.e.\ $x\in\mathbb{R}$, and square integrable with respect to $x$.
\end{prop}

\begin{proof}
It follows from Boufoussi et al.\  in \cite[Theorem 3.1]{BoufoussiDozziGuerbaz} that a stochastic process $X=\{X(t)\}_{t\in[0,T]}$ has a local time $L(t,x)$ that is continuous in $t$ for a.e.\ $x\in\mathbb{R}$, and square integrable with respect to $x$, if $X$ satisfies the following condition

$(\mathcal{H})$: There exist positive numbers $(\rho_{0},H)\in (0,\infty)\times (0,1)$ and a positive function $\psi\in L^{1}(\mathbb{R})$ such that for all $\kappa\in\mathbb{R}, T>0, t,s\in[0,T], 0<|t-s|<\rho_{0}$ we have
\begin{equation*}
\Bigg|\mathbb{E}\left[\exp\Big(i\kappa\frac{X(t)-X(s)}{|t-s|^{H}}\Big)\right]\Bigg|\leq \psi(\kappa).
\end{equation*}
\vskip10pt
\noindent Apply Lemma \ref{lem:TFBM and TFBMII bounds}, more precisely, the left-hand side of \eqref{two-sided1} and Remark \ref{rem2}, to get
\begin{equation*}
\begin{split}
\mathbb{E}\left[\exp\Big(i\kappa\frac{B^{I}_{H,\lambda}(t)-B^{I}_{H,\lambda}(s)}{|t-s|^{H}}\Big)\right]&=
\exp\Big(-|\kappa|^{2}\frac{\mathbb{E}[|{B}^{I}_{H,\lambda}(t)-{B}^{I}_{H,\lambda}(s)|^{2}]}{|t-s|^{2H}}\Big)\\
&\leq \exp\Big(-|\kappa|^{2}C\Big):=\psi(\kappa)
\end{split}
\end{equation*}
where the function $\psi\in L^{1}(\mathbb{R},d\kappa)$.  Hence TFBM satisfies condition $(\mathcal{H})$. Along the same line, using Lemma \ref{lem:TFBM and TFBMII bounds}, namely,  the left-hand side of \eqref{two-sided1} and Remark \ref{rem2}, we can see that
TFBMII satisfies condition $(\mathcal{H})$. Therefore, both $X=B^{I}_{H,\lambda}$ and  $X=B^{I\!I}_{H,\lambda}$ have the   local time $L(t,x)$ that is continuous in $t$ for a.e.\ $x\in\mathbb{R}$. The proof is completed.
\end{proof}
In the next step, we prove that tempered fractional processes are locally nondeterministic on any open interval $(0,T)$,     $T>0$. Recall that a zero mean Gaussian process $\{X(t)\}_{t\in\mathbb{R}}$ is \textit{locally nondeterministic} (LND) on some interval $\mathbb{T}=(a,b)$ if $ X$ satisfies condition $(A)$ consisting of the following three assumptions:

\begin{itemize}\label{condition A}
\item[$(A)$  $(i)$] $\mathbb{E}[ X^2(t) ] >0$ for all $t\in \mathbb{T}$;
\item[$(ii)$] $\mathbb{E}[(X(t)-X(s))^2]>0$ for all $t,s\in \mathbb{T}$ sufficiently close;
\item[$(iii)$] for any $m\geq 2$,
\begin{equation}\label{BermanVM}
\liminf_{\epsilon\downarrow 0}V_{m}=\frac{{\text{Var}}\{X(t_m)-X(t_{m-1})|X(t_1), \ldots, X(t_{m-1})\}}{{\text{Var}}\{X(t_m)-X(t_{m-1})\}}>0,
\end{equation}
where the $\liminf$ is taken over distinct, ordered $t_1<t_2<\ldots<t_m\in \mathbb{T}$ with $|t_1-t_m|<\epsilon$.
\end{itemize}

\begin{rem}
According to Berman \cite{Berman}, the ratio $V_{m}$ in assumption $(iii)$ is called a relative linear prediction error and is always between $0$ and $1$. If the ratio is bounded away from zero as $|t_{1}-t_{m}|\to 0$, then we can approximate $X(t_m)$ in the ${\mathbb{L}}^{2}$ norm by the most recent value $X(t_{m-1})$ with the same order of error as by the set of values $X(t_{1}),\ldots,X(t_{m-1})$.  We refer the reader to \cite{Berman} for more details.
\end{rem}

\begin{thm}\label{prop:TFBMIILND}
Let $X$ be either TFBM \eqref{eq:TFBMdef} or TFBMII \eqref{eq:TFBMIIdef}. Then  for any $0<H<1$ and  $\lambda>0$, $X$ is LND on every interval $(0,T)$ for $0<T<\infty$.
\end{thm}

\begin{proof}
By letting the index of stability $\alpha=2$ in the proof of Proposition 5.4 in \cite{Meerschaertsabzikar2}, one can prove that  TFBM is LND on every interval $(0,T)$ for $0<T<\infty$ (it is proved on any interval $(\delta, T)$ for $0<\delta<T<\infty$ but the proof does not refer to $\delta$ and can be extended to $(0,T)$). To prove  that TFBMII is LND on every   interval $(0,T)$, we need to verify assumptions $(i)$--$(iii)$ of condition $(A)$. The first and second assumptions follow immediately from the left-hand side of inequality \eqref{two-sided1}, Theorem \ref{lem:TFBM and TFBMII bounds}. It remains to show that the TFBMII $\{B^{I\!I}_{H,\lambda}(t)\}$ satisfies assumption $(iii)$.

 From \eqref{eq:TFBMIIdef} one can see that $\{B(u):u\leq s\}$ determines $\{B^{I\!I}_{H,\lambda}(u), u\leq s\}$ in the sense that
\begin{equation}\label{incl}
\sigma\left(B^{I\!I}_{H,\lambda}(u), u\leq s \right)\subset \sigma\left(B(u), u\leq s \right)
\end{equation}
for all $s>0$. So, for the moment, consider any  $s<t$ and  the value
$$ {\rm Var}\Big(B^{I\!I}_{H,\lambda}(t) - B^{I\!I}_{H,\lambda}(s)| B(u), u\leq s \Big).$$

Next, write the moving average representation in \eqref{eq:TFBMIIdef} as  follows:
\begin{equation*}
\begin{split}
B^{I\!I}_{H,\lambda}(t)&= \frac{1}{\Gamma(H+\frac{1}{2})}\left( \int_{-\infty}^{s}g^{I\!I}_{H,\lambda,t}(x) dB(x)
 +\int_{s}^{t}g^{I\!I}_{H,\lambda,t}(x)\ dB(x)\right),
\end{split}
\end{equation*}
and observe that
$\int_{-\infty}^{s}g^{I\!I}_{H,\lambda,t}(x) dB(x)$
is measurable with respect to $\sigma\left(B(u), u\leq s \right)$.
Therefore
\begin{equation}\label{var3}
\begin{split}
{\rm Var}\Big(B^{I\!I}_{H,\lambda}(t)|B(u), u\leq s \Big)
&={\rm Var}\Bigg( \frac{1}{\Gamma(H+\frac{1}{2})}\Big[\int_{s}^{t}
\big((t-x)_{+}^{H - \frac{1}{2}}e^{-\lambda (t-x)_{+}}\\
&+\lambda \int_{0}^{t} (w-x)_{+}^{H-\frac{1}{2}}e^{-\lambda(w-x)_{+}}\ dw\big)\ dB(x)\Big]\Big| B(u), u\leq s \Bigg)\\
&\geq \frac{1}{\Gamma^{2}(H+\frac{1}{2})} \int_{s}^{t}
  (t-x)^{2H-1}e^{-2\lambda(t-x)}   dx.
\end{split}
\end{equation}
Now, taking into account the form of the numerator in \eqref{BermanVM}, relation \eqref{incl} and the fact that finally the left-hand side of \eqref{var3} is bounded from below by some non-random value, we conclude that the relative predicted error $V_{m}$ is bounded from below by the following value:
\begin{equation}\label{bermanbound}
\frac{ \int_{s}^{t}   (t-x)^{2H-1}e^{-2\lambda(t-x)} \ dx}{\Gamma^{2}(H+\frac{1}{2})\text{Var}\Big(B^{I\!I}_{H,\lambda}(t)-B^{I\!I}_{H,\lambda}(s)\Big)}\ge \frac{     (t-s)^{2H}e^{-2\lambda(t-s)}}{\Gamma^{2}(H+\frac{1}{2})\text{Var}\Big(B^{I\!I}_{H,\lambda}(t)-B^{I\!I}_{H,\lambda}(s)\Big)},
\end{equation}
where $s=t_{m-1}$ and $t=t_{m}$.
Applying Lemma \ref{lem:TFBM and TFBMII bounds}  and Remark \ref{rem2}, we get that
\begin{equation}\label{eq:condition3TFBMII secondtpart}
{\text{Var}}\{B^{I\!I}_{H,\lambda}(t_m)-B^{I\!I}_{H,\lambda}(t_{m-1})\}\leq C_T\left|t_m-t_{m-1}\right|^{2H}
\end{equation}
for $|t_m-t_{m-1}|<\epsilon$ and all points being at interval $(0,T)$. Here $C_T$ is a constant depending only on $T$ but not on $m$ and the points $t_1, \ldots, t_m$. With the help of \eqref{eq:condition3TFBMII secondtpart}, we get that the ratio in \eqref{bermanbound} is bounded below by
\begin{equation*}
\frac{e^{-2\lambda(t_m-t_{m-1})}(t_m-t_{m-1})^{2H}}{2H \Gamma^{2}(H+\frac{1}{2})C_{2}(t_m-t_{m-1})^{2H}}=
\frac{e^{-2\lambda(t_m-t_{m-1})} }{2H \Gamma^{2}(H+\frac{1}{2})C_{T} }
\end{equation*}
for $|t_m-t_{m-1}|<\epsilon$ and all points being at interval $(0,T)$. This value tends to $ \frac{1 }{2H \Gamma^{2}(H+\frac{1}{2})C_{T} }$ as $\epsilon\downarrow 0$,
 and hence condition $(A)$ holds. It means $\{B^{I\!I}_{H,\lambda}\}$ is LND on $(0,T)$ and this completes the proof.
\end{proof}

\subsection{$p$-variation of tempered fractional processes}\label{sec3}

In this section, we show that TFBM and TFBMII have the same $\frac 1H$-variation as  the FBM, when $0<H<1$. First, we introduce  the "uniform" definition  of $\beta$-variation of a stochastic process. Let us introduce some notation.
Fix a time interval $[a,b]\subset \rr$, and consider the uniform partition
\begin{equation*}
\pi^{n}=\{a=t^{n}_{0}<t^{n}_{1}<\ldots<t^{n}_{n}=b\},
\end{equation*}
where $t^{n}_{i}=a+\frac{i}{n}(b-a)$ for $i = 0,\ldots,n$. Let $\beta\geq 1$ and $X=\{X_t,t  \in \rr\}$ be a continuous stochastic process. Moreover, we define $\Delta^{n}_{i} X=X(t^{n}_{i})-X(t^{n}_{i-1})$.

\begin{defn}
For any $\beta\ge 1 $ the $\beta$-variation of X on the interval $[a,b]$, denoted by $\langle X\rangle _{\beta,[a,b]}$, is the limit in probability of
\begin{equation*}
S^{[a,b]}_{\beta,n}(X):=\sum_{i=1}^{n}|\Delta^{n}_{i} X|^{\beta},
\end{equation*}
if the limit exists. We say that the $\beta$-variation
of $X$ on $[a,b]$ exists in $L^{p}$ if the above limit exists in $L^{p}$ for some $p\ge 1$.
\end{defn}
It is also easy to see that the following triangular inequality holds:
\begin{equation*}
S^{[a,b]}_{\beta,n}(X+Y)^{\frac{1}{\beta}}\leq S^{[a,b]}_{\beta,n}(X)^{\frac{1}{\beta}}+S^{[a,b]}_{\beta,n}(Y)^{\frac{1}{\beta}}.
\end{equation*}
This inequality implies that if $X$ and $Y$ are two continuous stochastic processes
such that $\langle X\rangle _{\beta,[a,b]}$ exists and $\langle Y\rangle _{\beta,[a,b]}=0$, then
\begin{equation}\label{eq:key formula for variation}
\langle X+Y\rangle _{\beta,[a,b]}=\langle X\rangle _{\beta,[a,b]}.
\end{equation}
Indeed, obviously $\langle X+Y\rangle_{\beta,[a,b]}\le\langle X\rangle_{\beta,[a,b]}$, and this inequality, complemented by the following one:
$$\langle X\rangle_{\beta,[a,b]}\le\langle X+Y\rangle_{\beta,[a,b]}+\langle -Y\rangle_{\beta,[a,b]},$$
immediately implies \eqref{eq:key formula for variation}.
Now, we are ready to state and proof the result of this section. The key for the proof of the result is \eqref{eq:key formula for variation} and using the well known fact that a normalized fBm has a finite  $\frac{1}{H}$-variation on any interval $[a,b]$, and it  equals $(b-a)\mathbb{E} [ |Z|^{\frac 1H}]$, where  $Z$ is a $\mathcal{N}(0,1)$-random variable, see,  e.g.  \cite{Mishura}, Section 1.18.

\begin{thm}\label{theo:variation} Let $X$ be either    a TFBM  $B_{H,\lambda}$ with  parameters $H\in (0,1)$ and $\lambda>0$, defined by \eqref{eq:TFBMdef} or a TFBMII $B^{I\!I}_{H, \lambda}$ given by \eqref{eq:TFBMIIdef}.  Then  $$\langle B_{H,\lambda}\rangle _{\frac 1H,[a,b]}=c_{H}(b-a)$$   in probability,
where $c_H= C(H)^{-\frac 1H}\mathbb{E} [ |Z|^{\frac 1H}]$ and $Z$ is a $\mathcal{N}(0,1)$-random variable.
 \end{thm}

\begin{proof}
The proof  for   a TFBM  and a  TFBMII is   similar and hence we only consider TFBM. First, apply the moving average representation of the TFBM  to get the decomposition
\[
B_{H,\lambda}(t)= B_{H}(t)+ Y(t),
\]
where
\[
B_{H}(t)=\int_{-\infty}^{+\infty}\big[(t-x)_{+}^{H-\frac{1}{2}} -(-x)_{+}^{H-\frac{1}{2}} \big]B(dx)
\]
and
\begin{equation}\label{eq:Y representation}
Y(t)=\int_{-\infty}^{+\infty}\big[(t-x)_{+}^{H-\frac{1}{2}}(e^{-\lambda(t-x)_{+}}-1)-(-x)_{+}^{H-\frac{1}{2}}(e^{-\lambda(-x)_{+}}-1)\big]B(dx)
\end{equation}
for $t\in \rr$.   Notice that  $C(H)^{-1} B_H$ is a fBm. Therefore, taking into account \eqref{eq:key formula for variation}, in order to prove the proposition, one needs to show that
\begin{equation*}
S^{[a,b]}_{\beta,n}(Y):=\sum_{i=1}^{n}|\Delta^{n}_{i} Y|^{\frac 1H},
\end{equation*}
converges to zero  in probability, where $\Delta^{n}_{i} Y=Y(t^{n}_{i})-Y(t^{n}_{i-1})$. In other words, we are in position to establish that $\langle Y\rangle _{\beta,[a,b]}=0$ where $Y$ is given by \eqref{eq:Y representation}. Obviously, the increments of $Y$ equal
\begin{equation*}
Y(t^{n}_{i+1})-Y(t^{n}_{i})
=\int_{-\infty}^{+\infty}\big[(t^{n}_{i+1}-x)_{+}^{H-\frac{1}{2}}(e^{-\lambda(t^{n}_{i+1}-x)_{+}}-1)-(t^{n}_{i}-x)_{+}^{H-\frac{1}{2}}(e^{-\lambda(t^{n}_{i}-x)_{+}}-1)\big]B(dx)
\end{equation*}
and then
\begin{equation*}
\begin{split}
&\sum_{i=1}^{n}\mathbb{E}\Big(|Y(t^{n}_{i+1})-Y(t^{n}_{i})|^{\frac 1H}\Big)\\
&=C\sum_{i=1}^{n}\Big(\int_{-\infty}^{\infty}\Big[(t^{n}_{i+1}-x)_{+}^{H-\frac{1}{2}}(e^{-\lambda(t^{n}_{i+1}-x)_{+}}-1)-(t^{n}_{i}-x)_{+}^{H-\frac{1}{2}}(e^{-\lambda(t^{n}_{i}-x)_{+}}-1)                             \Big]^{2}\ dx\Big)^{\frac{1}{2H}}\\
&\leq C\sum_{i=1}^{n}\Big(\int_{-\infty}^{t^{n}_{i}}\Big[(t^{n}_{i+1}-x)^{H-\frac{1}{2}}(e^{-\lambda(t^{n}_{i+1}-x)}-1)-(t^{n}_{i}-x)^{H-\frac{1}{2}}(e^{-\lambda(t^{n}_{i}-x)}-1)\Big]^{2}\ dx\Big)^{\frac{1}{2H}}\\
& \quad +C\sum_{i=1}^{n}\Big(\int_{t^{n}_{i}}^{t^{n}_{i+1}}\Big[(t^{n}_{i+1}-x)^{H-\frac{1}{2}}(e^{-\lambda(t^{n}_{i+1}-x)}-1)\Big]^{2}\ dx\Big)^{\frac{1}{2H}}\\
&=: C(I_{1,n}+I_{2,n}),
\end{split}
\end{equation*}
where $C$ is a generic constant that   depends on $H$.
Let us first show that $I_{2,n}\to 0$ as $n\to\infty$. Using the change of variable $t^{n}_{i+1}-x=y$ in $I_{2}$ and  the inequality $|e^{-a}-e^{-b}|\leq |a-b|$ for $a, b>0$ we can write
\begin{equation*}
\begin{split}
I_{2,n}&= n\Big(\int_{0}^{\frac{b-a}{n}}y^{2H-1}(e^{-\lambda y}-1)^{2}\ dy\Big)^{\frac{1}{2H}}
 \leq n\lambda^{\frac 1H }\Big(\int_{0}^{\frac{b-a}{n}}y^{2H+1}\ dy\Big)^{\frac{1}{2H}}= C n^{-\frac 1H }\to 0
\end{split}
\end{equation*}
as $n\to\infty$. Next, we show that $I_{1,n}\to 0$ as $n\to\infty$. First we use the change of variable $t^{n}_{i}-x=y$   to see that
\begin{equation*}
 I_{1,n} =n\left(\int_{0}^{\infty}\left[\left(y+\frac{b-a}{n}\right)^{H-\frac{1}{2}}\left(1-e^{-\lambda(y+\frac{b-a}{n})}\right)-y^{H-\frac{1}{2}}\left(1-e^{-\lambda y}\right)\right]^{2}\ dy\right)^{\frac{1}{2H}},
\end{equation*}
and it is sufficient to prove that
\begin{equation*}
I^{2H}_{1,n}=n^{2H}\int_{0}^{\infty}\left[\left(y+\frac{b-a}{n}\right)^{H-\frac{1}{2}}\left(1-e^{-\lambda(y+\frac{b-a}{n})}\right)-y^{H-\frac{1}{2}}\left(1-e^{-\lambda y}\right)\right]^{2}\ dy\rightarrow 0
\end{equation*}
as $n\rightarrow \infty$.
Further,
\begin{equation*}
\begin{split}
 I^{2H}_{1,n}&\le  2n^{2H}\left(1-e^{-\lambda(\frac{b-a}{n})}\right)^2\int_{0}^{\infty} \left(y+\frac{b-a}{n}\right)^{2H-1}e^{-2\lambda y}\ dy\\
 &\quad +2(b-a)^{2H}\int_{0}^{\infty}\left((z+1)^{H-1/2}-z^{H-1/2}\right)^2(1-e^{-\lambda\left(\frac{b-a}{n}\right)z})^2dz\\
&=I^{2H}_{11,n}+2(b-a)^{2H}I^{2H}_{12,n},
\end{split}
\end{equation*}
where in the second integral we changed the variable $y=\frac{b-a}{n}z.$
It is easy to see that for $n>b-a$
\begin{equation*}
\begin{split}
I^{2H}_{11,n}&\le 2n^{2H-2}(\lambda(b-a))^2\int_{0}^{\infty} \left(\left(y+1\right)^{2H-1}\vee y^{2H-1}\right)e^{-2\lambda y}\ dy\rightarrow 0
\end{split}
\end{equation*}
as $n\rightarrow \infty$.
Concerning $I^{2H}_{12,n}$, we observe that $(1-e^{-\lambda\left(\frac{b-a}{n}\right)z})^2\rightarrow 0$ as $n\rightarrow \infty$, an splitting $I^{2H}_{12,n}=\int_{0}^{1}+\int_{1}^{\infty}$, we immediately get that $\int_{0}^{1}\rightarrow 0$ as $n\rightarrow \infty$, while the integrand in the $\int_{1}^{\infty}$ can be bounded as follows:
\begin{equation*}
\begin{split}\left((z+1)^{H-1/2}-z^{H-1/2}\right)^2(1-e^{-\lambda\left(\frac{b-a}{n}\right)z})^2&\le
(H-1/2)^2\left((z+1)^{2H-3}\vee  z^{2H-3}\right),
\end{split}
\end{equation*}
and $\int_{1}^{\infty}$ converges to zero due to the Lebesgue   dominated convergence theorem. Now the proof is complete.

\end{proof}
\begin{rem} Theorem \ref{theo:variation} implies immediately that $p$-variation of a TFBM  and a  TFBMII   equals   zero or infinity, depending on whether $p$ is greater or less than $1/H$.\end{rem}

\section{Breuer--Major theorem in application to tempered fractional Gaussian noises}\label{sec:BM}

In this section, we consider   the tempered fractional Gaussian noises in the context of  popular   Breuer-Major Theorem  (see \cite{Breuer} or \cite[Theorem 7.2.4]{Nourdin-Peccati-Bible} for a modern exposition) in the Gaussian analysis.
\subsection{Covariance structures of tempered fractional Gaussian noises}
First,  we study the increment of tempered fractional Gaussian processes and investigate the asymptotic behavior of the their covariance functions for large lags. These  results then provide  a useful tool to develop some limit theorems. For simplicity, denote  $\alpha=H-\frac{1}{2}$. Given a TFBM \eqref{eq:TFBMdef}, we define tempered fractional Gaussian noise (TFGN)
\begin{equation*}\label{eq:TFGNdef}
X^{I}_{\alpha,\lambda}(j)=B^{I}_{H,\lambda}(j+1)-B^{I}_{H,\lambda}(j)\quad\text{for}\quad  j\in \mathbb{Z}.
\end{equation*}
It follows easily from \eqref{eq:TFBMdef} that TFGN has the moving average representation:
\begin{equation}\label{eq:TFGNmoving}
X^{I}_{\alpha,\lambda}(j)= \int_{\mathbb{R}}g^{I}_{\lambda,\alpha,j}(x)B(dx)=
{\int_{\mathbb{R}}\left[e^{-\lambda(j+1-x)_{+}}(j+1-x)_{+}^{\alpha}-e^{-\lambda(j-x)_{+}}(j-x)_{+}^{\alpha} \right]B(dx)}.
\end{equation}
Along the same lines, a tempered fractional Gaussian noise of the second kind (TFGNII) can be defined as follows:
\begin{equation*}
X^{I\!I}_{\alpha,\lambda}(j)=B^{I\!I}_{H,\lambda}(j+1)-B^{I\!I}_{H,\lambda}(j)\quad\text{for  $j\in \mathbb{Z}$.}
\end{equation*}
It follows from  \eqref{hdef0} that  a TFGNII  has the moving average representation
\begin{equation}\label{eq:TFGNIImoving}
\begin{split}
X^{I\!I}_{\alpha,\lambda}(j)&= \int_{\mathbb{R}}g^{I\!I}_{\lambda,\alpha,j}(x)B(dx)=\int_\rr \Big[ e^{-\lambda(j+1-x)_{+}}(j+1-x)_{+}^{\alpha}-e^{-\lambda(j-x)_{+}}(j-x)_{+}^{\alpha}\\
&\qquad\qquad\qquad\quad+\lambda \int_{j}^{j+1} e^{-\lambda(s-x)_{+}}(s-x)_{+}^{\alpha} ds\Big]  B(dx).
\end{split}
\end{equation}

So, let $X^{I}_{\alpha,\lambda}(j)$ and $X^{I\!I}_{\alpha,\lambda}(j)$  be the stationary sequences given by \eqref{eq:TFGNmoving} and \eqref{eq:TFGNIImoving} respectively. Denote
\begin{equation}\label{eq:covar}\gamma^J (k):=\mathbb{E}[X^{J}_{\alpha,\lambda}(0) X^{J}_{\alpha,\lambda}(k)]=\abs{k+1}^{2H}(C^{J}_{\vert k \vert +1})^2-2\abs{k}^{2H} (C^{J}_{\vert k\vert})^2+\abs{k-1}^{2H}(C^{J}_{\vert k-1 \vert })^2,\; J=I,II,\end{equation}
 where the normalizing constants $C^J_t$ are presented in Lemma \ref{lem:usefulrelations}. To analyze the behavior of $\gamma^J (k)$, we shall combine  its direct representation via the kernels $g^{J}_{\lambda,\alpha}$, $J=I,II $ and its representation from \eqref{eq:covar}. The following   lemma  specifies the behavior of the intermediate noise covariance and will be used in the proof of the main theorems.
\begin{lem}\label{lem:Positive-Negative-Correlation}
	\begin{itemize}
		\item[(a)] Let $\lambda >0$. Consider function $\psi (t) = (C^I_t)^2  \, t^{2H}$ for $t>0$ where the normalizing constant $C^I_t$ is given in Lemma \ref{lem:usefulrelations}. Then $\psi''(t)<0$ for all $t>0$ provided that $H \in (0,\frac{1}{2}]$. Hence, TFGN  is negatively correlated when $H \in (0,\frac{1}{2}]$ meaning that for every $0 \neq k \in \Z$
		\begin{equation}\label{eq:negative-correlation-I}
		\gamma^I (k)  < 0.
		\end{equation}
	\item[(b)] Let $ \lambda >0$. Then  for every $ k \in \Z$ and $H>1/2$,
		\begin{equation}\label{eq:positive-correlation-II}
		\gamma^{I\!I} (k)  >0.
		\end{equation}
	Moreover, when $H=1/2$, it holds $\gamma^{I\!I} (0)  >0$, and $\gamma^{I\!I} (k)  =0$ for every $0 \neq k \in \Z$.
		\end{itemize}
	\end{lem}

\begin{proof}
(a) Note that using Lemma \ref{lem:usefulrelations} we can write
\[  \psi (t) = \frac{2\Gamma(2H)}{(2\lambda)^{2H}}-\frac{2\Gamma(H+\frac 12)}{\sqrt{\pi} (2\lambda)^{H}}t^H K_{H}(\lambda t)= \frac{2\Gamma(2H)}{(2\lambda)^{2H}} - \frac{2\Gamma(H+\frac 12)}{\sqrt{\pi} 2^{H}\lambda^{2H}} (\lambda t)^{H} K_{H}(\lambda t). \] Hence, using relation  $\frac{d}{dx} (x^\nu K_\nu (x))= - x^\nu K_{\nu -1}(x)$ for all $\nu \in \R$ (see e.g., Appendix in \cite{Robert-Bessel-functions}), we can immediately deduce that
\begin{align}
\psi'(t) & =  \frac{2\Gamma(H+\frac 12)}{\sqrt{\pi} 2^{H}\lambda^{2H}} \lambda (\lambda t)^{H} K_{H-1} (\lambda t) = \frac{2\Gamma(H+\frac 12)}{\sqrt{\pi} 2^{H}\lambda^{2H-2}} \big\{ t (\lambda t)^{H-1}K_{H-1}(\lambda t) \big\},\\
\psi''(t) & = \frac{2\Gamma(H+\frac 12)}{\sqrt{\pi} 2^{H}\lambda^{2H-2}}  \Big\{    (\lambda t)^{H-1}K_{H-1}(\lambda t) - \lambda t (\lambda t)^{H-1} K_{H-2}(\lambda t) \Big\} \nonumber\\
& = \frac{2\Gamma(H+\frac 12)}{\sqrt{\pi} 2^{H}\lambda^{2H-2}} (\lambda t)^{H-1} \Big\{ K_{H-1}(\lambda t) - (\lambda t) K_{H-2} (\lambda t) \Big\}\\
&  = -\frac{2\Gamma(H+\frac 12)}{\sqrt{\pi} 2^{H}\lambda^{2H-2}} (\lambda t)^{H-1}  \Big(  (\lambda t) K_{H-1}(\lambda t) \Big) \times \Big( \frac{K_{H-2}(\lambda t)}{K_{H-1}(\lambda t)}  - \frac{1}{\lambda t} \Big).
\end{align}
It is well known that $K_\nu (x) >0$ for every $x>0$ and real $\nu \in \R$. Let $\nu=H-1$. Therefore, it is enough to understand  the sign of the quantity,
\begin{equation}\label{eq:T1-H<1/2-sign}
f(x):= \frac{K_{\nu-1}(x)}{K_\nu (x)}  - \frac{1}{x}.
\end{equation}

 Let $\nu< - 1/2$, or equivalently $H<1/2$.  In this case \cite[Proposition 4.5]{Bessel-Ratio-PAMS} contains  a sharp estimate that can be used to rewrite function $f$ as
\[ f(x)= \frac{K_{\nu-1}(x)}{K_\nu (x)}  -\frac{1}{x} =  \frac{K_{\mu+1}(x)}{K_\mu (x)} - \frac{1}{x} > 1, \quad \forall \, x >0,  \]
where $\mu=-\nu > 1/2$. Hence, function $f$ stays positive over the whole interval $(0,\infty)$. This means that the  noise TFGN is globally negatively correlated. For $H=1/2$ situation is very simple: covariance function equals
$$\E B^I_{1/2,\lambda}(t)B^I_{1/2,\lambda}(s)=\frac{1}{2\lambda}\left(e^{-\lambda|t-s|}-e^{-\lambda t}-e^{-\lambda s}+1\right),$$
whence $$\gamma^I (k)=\frac{1}{\lambda} e^{-\lambda|k|}\left(1-\cosh \lambda\right)<0. $$ It also means that the  noise TFGN is globally negatively correlated.

(b)  Let $H>1/2$. Integrating by parts,  we can rewrite the kernel $g^{I\!I}_{\lambda,\alpha,j}$ as $$g^{I\!I}_{\lambda,\alpha,j}(x)=\alpha\int_{j}^{j+1} e^{-\lambda(s-x)_{+}}(s-x)_{+}^{\alpha-1} ds>0,$$
and the proof immediately follows. When $H=1/2$, the tempered fractional Brownian motion of the second kind coincides with a Brownian motion, and hence the claim follows at once.

\end{proof}
\begin{rem}
	Item (a) from  Lemma \ref{lem:Positive-Negative-Correlation} reveals that TFGN  is globally negatively correlated provided that $H \in (0,1/2]$ no matter what tempering parameter $\lambda$ is. However,  when Hurst parameter $H > 1/2$ a switching regime  takes place that can be useful in modeling. More precisely, certainly there are time points $t^* = t^* (H,\lambda)  \le t^{**}=t^{**}(H,\lambda)$ so that TFGN  is positively correlated for every continuous lags $t< t^*$ and negatively correlated for all the lags $t>t^{**}$.  Although, we were unable to prove that one can take $t^* = t^{**}$ however, our numerical MATALAB observations illustrate that this is in fact the case.  The main obstacle in front to verify the uniqueness of time point where the TFGN   switches from positive correlation to negative correlation is to show the strict monotonicity (increasing) of the function  $$\frac{K_{\mu+1}(x)}{K_{\mu} (x)}  - \frac{1}{x}$$  over the interval $(0,1)$ provided that $\mu \in [0,1/2)$.
	As in the Breuer--Major Theorem we are interested in the behavior of the noise in the discrete clock, therefore positions of the critical times $t^* $ and $t^{**}$ are very significant, and that heavily depends on the Hurst parameter $H>0$ as well as the tempering parameter $\lambda>0$. For example, when $H = 3/2$, it can be shown that $t^* = t^{**} = \frac{1}{\lambda}$, and  that
	\begin{align}
	\psi''(t) & > 0, \quad  \text{ for } \quad t \in (0,\frac{1}{\lambda}),\\
	\psi''(t) & <0, \quad  \text{ for } \quad t > \frac{1}{\lambda}.
	\end{align}
	Hence, TFGN admits at the same time positive and negative correlation of discrete lags depending on the range of tempering parameter $\lambda$.  We also feel that similar switching regime phenomenon takes places for TFGNII   when $H <1/2$.

\end{rem}
Now we are in position to investigate the asymptotic behavior of the increments of TFGN and TFGNII at large lags.

\begin{prop}\label{prop:TFLN}
 Then we claim the following asymptotic behavior of covariances.

\begin{itemize}
\item[$(a)$]  For any  $\alpha > -\frac{1}{2}$,
\begin{equation}\label{x-2}\gamma^{I}(j)\sim -\frac{2\Gamma(\alpha+1)( \cosh \lambda-1 )}{  (2\lambda)^{\alpha+1}} e^{-\lambda j} j^{\alpha} \end{equation} as $j\to\infty$. It means that asymptotically TFGN has   negative correlation for any $\alpha > -\frac{1}{2}$ (compare to Lemma \ref{lem:Positive-Negative-Correlation}).  In particular, $\gamma^I \in \ell^q(\Z) $ for every $q \ge 1$.

\item[$(b)$]  For any  $\alpha > -\frac{1}{2}$,  $$\gamma^{I\!I}(j)  \sim (2e^\lambda-1) (2\lambda)^{-\alpha-1}  \Gamma(\alpha+1) e^{-\lambda j } j^{ \alpha -1 }$$ as    $j\to\infty$. It means that asymptotically TFGNII has   positive correlation (compare to Lemma \ref{lem:Positive-Negative-Correlation}).   In particular, $\gamma^{I\!I} \in \ell^q(\Z) $ for every $q \ge 1$.
 \end{itemize}
\end{prop}

\begin{proof}
 \begin{itemize}
\item[$(a)$] The following transformations are immediate:
\begin{align*}
\gamma^{I}(j)&=\mathbb{E}\bigg(\int_{-\infty}^{j+1}\left(e^{-\lambda(j+1-x)}(j+1-x)^{\alpha} -e^{-\lambda(j-x)_+}(j-x)^{\alpha}_+\right)dB(x)\\
& \quad \times\int_{-\infty}^{1}\left(e^{-\lambda(1-x)}( 1-x)^{\alpha} -e^{-\lambda( -x)_+}( -x)^{\alpha}_+\right)dB(x)\bigg)
\\
&=\int_{-\infty}^{1}\left(e^{-\lambda(j+1-x)}(j+1-x)^{\alpha} -e^{-\lambda(j-x)}(j-x)^{\alpha}\right)\\
&\quad \times \left(e^{-\lambda(1-x)}( 1-x)^{\alpha} -e^{-\lambda( -x)_+}( -x)^{\alpha}_+\right)d x \\
&=e^{-\lambda j}\bigg(\int_{0}^{\infty} e^{-2\lambda z}z^{\alpha}\left( (j+z)^{\alpha} -e^\lambda(j-1+z)^{\alpha}\right)dz\\
& \quad -
\int_{0}^{\infty} e^{-2\lambda z}z^{\alpha}\left( e^{-\lambda}(j+1+z)^{\alpha} -(j+z)^{\alpha}\right)dz\bigg) \\
&=e^{-\lambda j}j^{\alpha}\int_{0}^{\infty}e^{-2\lambda z}z^{\alpha}\bigg(2\left(1+\frac{ z}{j}\right)^{\alpha}- e^{-\lambda}\left(1+\frac{z+1}{j}\right)^{\alpha}-e^{\lambda}\left(1+\frac{z-1}{j}\right)^{\alpha}\bigg)dz.
\end{align*}
  Consider the value in the bracket $$2\left(1+\frac{ z}{j}\right)^{\alpha}- e^{-\lambda}\left(1+\frac{z+1}{j}\right)^{\alpha}-e^{\lambda}\left(1+\frac{z-1}{j}\right)^{\alpha}.$$
  It tends to  $2-e^\lambda-e^{-\lambda}$ as $j\rightarrow\infty$, and for $j\geq 2$ is bounded by
  $$(2\left(1+ { z} \right)^{\alpha} +  e^{-\lambda}\left(2+ z \right)^{\alpha}+e^{\lambda}z^{\alpha})\vee(2+2e^\lambda+e^{-\lambda}).$$
  It means that we can apply Lebesgue dominated convergence theorem and get $(a)$.
\item[$(b)$] Denote   $$g_j(x):=e^{-\lambda(j+1-x)_{+}}(j+1-x)_{+}^{\alpha}-e^{-\lambda(j-x)_{+}}(j-x)_{+}^{\alpha}
 +\lambda \int_{j}^{j+1} e^{-\lambda(s-x)_{+}}(s-x)_{+}^{\alpha} ds,$$
 then, by the similar calculations as in the part (a),  $\gamma^{I\!I}(j)=\int_{-\infty}^{1}g_j(x)g_0(x)dx$.
 So, our goal is to study the asymptotic behavior of $g_j(x)$. Note that on the interval $(-\infty,1]$
 \begin{equation*}
\begin{split}
g_j(x)&=e^{-\lambda(j+1-x)}(j+1-x)^{\alpha}-e^{-\lambda(j-x)}(j-x)^{\alpha}
 +\lambda \int_{j}^{j+1} e^{-\lambda(s-x)}(s-x)^{\alpha} ds\\&=e^{-\lambda j}j^\alpha\left(e^{-\lambda(1-x)}\left(1+\frac{1-x}{j}\right)^{\alpha}-e^{\lambda x }\left(1-\frac{x}{j}\right)^{\alpha}
 +\lambda \int_{0}^{1} e^{-\lambda(z-x)}\left(1+\frac{z-x}{j}\right)^{\alpha} dz\right).
\end{split}
\end{equation*}
Applying Taylor expansion to the terms $\left(1+\frac{1-x}{j}\right)^{\alpha}, \left(1-\frac{x}{j}\right)^{\alpha}$ and
$\left(1+\frac{z-x}{j}\right)^{\alpha}$, and integrating the last integral by parts, we get that

\begin{equation*}
\begin{split}
 g_j(x)&=e^{-\lambda j}j^\alpha\bigg(e^{-\lambda(1-x)}\left(1+\alpha\frac{1-x}{j}\right) -e^{\lambda x }\left(1-\alpha\frac{x}{j}\right)
 \\&+\lambda \int_{0}^{1} e^{-\lambda(z-x)}\left(1+\alpha\frac{z-x}{j}\right)  dz\bigg)+h_j(x)=
 \alpha\lambda^{-1}(1-e^{-\lambda}) e^{-\lambda j}j^{\alpha-1}e^{\lambda x}+h_j(x),
 \end{split}
 \end{equation*}
where  $h_j(x)=j^{-2}h(x)$, and $h(x)$ is, up to a constant multiplier, of order $e^{\lambda x}x^2$.
Applying again the  Lebesgue dominated convergence theorem, we get that
$$\gamma^{I\!I}(j)\sim \alpha\lambda^{-1}(1-e^{-\lambda}) e^{-\lambda j}j^{\alpha-1}\int_{-\infty}^{1} e^{\lambda x}g_0(x)dx.$$
As regards the last integral, it equals
\begin{equation*}
\begin{split}\int_{-\infty}^{1} e^{\lambda x}g_0(x)dx&=\int_{-\infty}^{1} e^{\lambda x}\left( e^{-\lambda(1-x)}(1-x)^{\alpha}-e^{-\lambda(-x)_{+}}(-x)_{+}^{\alpha}
 +\lambda \int_{x}^{1} e^{-\lambda(s-x)}(s-x)^{\alpha} ds\right)dx\\&=(e^\lambda-1)\int_{0}^{\infty}e^{-2\lambda z}z^\alpha dz+\lambda e^\lambda\int_{0}^{\infty}e^{-\lambda u}\int_0^ue^{-\lambda z}z^\alpha dzdu\\&
 =(2e^\lambda-1) (2\lambda)^{-\alpha-1}  \Gamma(\alpha+1),
 \end{split}
 \end{equation*}
 and the proof follows.
 \end{itemize}
 \end{proof}

\begin{lem}\label{Hermite} Let $Y^{I}(j)=(C^I_1)^{-1}X^{I}_{\alpha,\lambda}(j)$  and $Y^{I\!I}(j)=(C_1^{I\!I})^{-1}X^{I\!I}_{\alpha,\lambda}(j)$ be normalized tempered fractional Gaussian noises with associated normalizing constants $C^I_1$ and $C^{I\!I}_1$, appearing in Lemma \ref{lem:usefulrelations}, and covariance functions $\gamma^I$ and $\gamma^{I\!I}$, respectively. Let $V^J_{n,q}=\frac{1}{\sqrt{n} }\sum_{k=1}^{n} H_q(Y^J(k)), J=I,I\!I $, where $H_q$ stands for the $q$th Hermite polynomial.
Then  \begin{equation}\begin{gathered}\label{orieq}
\sigma^2_{n,J,q} : = \Var \left( V^J_{n,q}\right) = \frac{q!}{n} \,   (C_1^J)^{-2q}  \sum_{k,l=1}^{n} \left(\gamma^J (k-l)\right)^q\longrightarrow \sigma^2_{J,q,H,\lambda}:=  q!  \,  (C_1^J)^{-2q}  \sum_{k \in \Z} \left(\gamma^J (k)\right)^q < +  \infty.
\end{gathered}\end{equation}
  Furthermore we can guarantee this value is strictly positive provided that
\begin{itemize}
	\item[(a)] $J=I,I\!I $ and  $q$ is even.
	\item[(b)] $J=I$, $H\in (0,1/2]$ and $q >1$.
	\item[(c)] $J=I\!I$, $H \ge 1/2$ and $q >1$.
	\end{itemize}	

\end{lem}
 \begin{proof} Finiteness of the sum $\sum_{k \in \Z}\left|\gamma^I (k)\right|^q$ follows from Proposition \ref{prop:TFLN}. The first equality in the relation \eqref{orieq} is mentioned, e.g.,  in the proof of  Theorem 7.2.4 \cite{n-p-AOP-Exact-Asymptotic}. Obviously,  $$\sigma^2_{n,J,q}  = q! \, (C^J_1)^{-2q}   \sum_{\vert k \vert < n}  \left(1 - \frac{\vert k \vert }{n}\right) \left(\gamma^J (k)\right)^q,$$
 this sum is nonnegative, and for even $q$ the value $\sigma^2_{n,J,q}$ strictly increases in $n$ therefore its limit is strictly positive.  For odd  $q$   we can state that the limit exists due to the dominated convergence theorem and  finiteness of the sum $\sum_{k \in \Z}\left|\gamma^I (k)\right|^q$, and this limit is obviously   nonnegative.
 For strict positivity of the limiting variance, part (a) is clear. (b)  First note that Proposition \ref{prop:TFLN} part (a) yields that $\gamma^I \in L^1 (\Z)$ is an absolutely convergent sum, and hence by a telescopic argument, we can write
 	\begin{equation}\label{eq:>0-1}
 	\sum_{k \in \Z} \frac{\gamma^I (k)}{(C^I_1)^2} = 1+ 2 \sum_{k \ge 1} \frac{\gamma^I (k)}{(C^I_1)^2} =0, \quad  \Longrightarrow \quad  \sum_{k \ge 1} \frac{\gamma^I (k)}{(C^I_1)^2}  = -1/2.
 	\end{equation}
 Now Lemma \ref{lem:Positive-Negative-Correlation} item (a) implies that $ \frac{\gamma^I (k)}{(C^I_1)^2} \in (-1,0)$ for every $0 \neq k \in \Z$. Let $q>1$ be an arbitrary integer. Then $$ \left(  \frac{\gamma^I (k)}{(C^I_1)^2} \right)^q >  \frac{\gamma^I (k)}{(C^I_1)^2}, \qquad  k \ge 1.$$ Therefore,
 \begin{align*}
 \sum_{k\in\Z}     \frac{(\gamma^I (k))^q}{(C^I_1)^{2q}}   =  1+ 2 \sum_{k \ge 1}   \frac{(\gamma^I (k))^q}{(C^I_1)^{2q}} > 1 + 2 \sum_{k \ge 1}   \frac{\gamma^I (k)}{(C^I_1)^{2}}=0.
 \end{align*}
 (c) It is clear due to Lemma \ref{lem:Positive-Negative-Correlation}, part (b).
 \end{proof}

\begin{rem}
Meerschaert and Sabzikar \cite[Remark 4.1]{Meerschaertsabzikar} pointed out that the covariance function $\gamma^{I}$  of TFGN of the first kind behaves  asymptotically as $\frac{2H(2H-1)\Gamma(2H)}{(2\lambda)^{2H}}j^{-2}$ for large lags $j$. However, part (a) of Proposition \ref{prop:TFLN} carefully shows that the asymptotic behavior of TFGN should be represented as $-\frac{2\Gamma(H+1/2)( \cosh \lambda-1 )}{  (2\lambda)^{H+1/2}} e^{-\lambda j} j^{H-1/2}$ for large lags  $j$.
\end{rem}

 \subsection{CLTs for the tempered fractional Gaussian noise processes}
As an application of the analysis of the behavior of the noise  covariance, we can derive the   CLT for the tempered fractional Gaussian noise processes.
Our first result treats the Gaussian fluctuations of the tempered fractional Gaussian noises in the setup of the Breuer--Major theorem. 
\begin{thm}[Breuer--Major theorem for tempered fractional Gaussian noises]\label{thm:Breuer-Major_tempered}
Let $\gamma(dx) = \frac{1}{\sqrt{2\pi}}e^{-x^2/2} dx$ denote the standard Gaussian measure on the real line. Assume that $f \in L^2(\R,\gamma)$ be a centered function, i.e. $\E_\gamma[f]=0$, with  Hermite rank  $d \ge 1$, meaning that, $f$ admits the Hermite expansion $f(x) = \sum_{j=d}^{\infty} a_j H_j(x)$ with $a_d \neq 0$.
   We have that
\begin{equation*}
V^J_{n,d,H,\lambda} := \frac{1}{\sqrt{n}}\sum_{k=1}^{n}f(Y^{J}_j) \limd \mathcal{N}(0, {\sigma^{2}_{J,H,\lambda,d}}),
\end{equation*}
with
\begin{equation}\label{eq:Target-Variance-TemperedI}
\sigma^{2}_{J,H,\lambda,d}=\sum_{q=d} ^\infty\frac{q!}{2^{q}} a_q^2\sigma^2_{J,q,H,\lambda} \in [0,\infty),
\end{equation}
where $\sigma^2_{J,q,H,\lambda}$ is introduced in \eqref{orieq}.

In any of the following cases: (a) $J=I,II$ and  $a_q\neq0$ for at least one of  even $q$; (b) $J=I $ and $H\leq 1/2$;
(c) $J=II $ and  {$H\ge1/2$} we claim that $\sigma^{2}_{J,H,\lambda,d}>0$.
\end{thm}

\begin{proof}
Consider $J=I$. First note that by applying part $(a)$ of Proposition \ref{prop:TFLN}, for every fixed $H,\lambda>0$, we have $\gamma^I \in \l^{p}(\Z)$ for all $p>0$, and in particular $\gamma^I \in \l^{d}(\Z)$ where $d$ denotes the Hermite rank. As a direct consequence, the classical Breuer--Major Theorem \ref{thm:Breuer-Major} can be applied, and in order to obtain the desired result, we are only left to compute the limiting variance.   Next, it is standard that (see e.g.,  \cite[page 131]{Nourdin-Peccati-Bible}) the dominated convergence theorem yields as $n$ tends to infinity,
  \begin{eqnarray*}
   \sigma^2_n &: = &\Var\left(V^J_{n,d,H,\lambda}  \right) = \sum_{j=d}^{\infty} j! a^2_j C^{-2j}_1   \frac{1}{n} \, \sum_{k,l=1}^{n} \gamma^I(k-l)^j
   = \sum_{j=d}^{\infty} j! a^2_j C^{-2j}_1 \sum_{\vert k \vert < n}  \left( 1 - \frac{\vert k \vert }{n} \right) \gamma^I(k)^j\\
& \longrightarrow & \sum_{j=d}^{\infty} j! a^2_j \sum_{k\in \Z} \left( C^{-2}_1 \gamma^I(k) \right)^j=: \sigma^2_{I,H,\lambda,d}.
   \end{eqnarray*}

Recall that $\vert C^{-2}_1 \gamma^I(k)\vert \le 1$ for all $k \in \Z$, and therefore one can readily infer that
\[  \sigma^2_{I,H,\lambda,d} = \sum_{j=d}^{\infty} j! a^2_j C_1^{-2j} \sum_{k\in\Z} \gamma^I(k)^j \le  C_1^{-2d} \,  \sum_{j=d}^{\infty} j! a^2_j \sum_{k\in\Z} \vert  \gamma^I(k)\vert^d  \le C_1^{-2d} \,  \Vert f \Vert^2_{L^2(\R,\gamma)} \sum_{k\in\Z} \vert  \gamma^I(k)\vert^d < +\infty. \]
Now the proof immediately follows from Lemma \ref{Hermite}.
 \end{proof}

\begin{rem}
\begin{itemize}
\item[(i)] The message of Theorem \ref{thm:Breuer-Major_tempered} is that tempering always fulfills the sufficient condition in the Breuer--Major Theorem  without assuming any extra condition on the Hurst parameter $H$ or/and the tempering parameter $\lambda$. This is in contract to the classical setup of the fractional Gaussian noise where often there is a phase transition for the validity of CLT, see \cite[Theorem 7.4.1]{Nourdin-Peccati-Bible}.
\item[(ii)] In fact, according to the second Dini's theorem \cite[Proposition C.3.2]{Nourdin-Peccati-Bible} convergences in parts $(a)$ and $(b)$ of Theorem \ref{thm:Breuer-Major_tempered} holds in the Kolmogorov topology too. Furthermore, one can show the convergence  holds in more stronger topology under some mild assumption on function $f$. This is topic of the forthcoming result.
    \end{itemize} 	
	\end{rem}

The next result aims to provide a quantitative version of the aforementioned CLTs. For given random elements $F$ and $G$ the \textit{total variation} distance, denoted by $d_{TV}$, between the laws of $F$ and $G$ defined as
\[  d_{TV} (F,G):= \sup_{A} \Big \vert \Pp (F \in A) - \Pp(G \in A) \Big \vert\]
where the supremum is taken over all the Borel subsets $A \in \mathcal{B}(\R)$ on the real line. Also, we introduce the Sobolev space $\mathbb{D}^{p,k}(\R,\gamma)$, where $p \ge 1$ and $k \in \N$, that is the closure of all polynomial mapping $f:\R \to \R$ with respect to the norm
\[ \Vert f \Vert_{p,k} := \Bigg[  \sum_{i=0}^{k} \int_{\R}\vert f^{(i)} (x) \vert^p \gamma (dx) \Bigg]^{\frac{1}{p}}.  \]
Here $f^{(0)}=f$, and $f^{(i)}$ stands for the $i$th derivative  of $f$, $i=1,\ldots, k$.

\begin{thm}\label{thm:quantitative-TV-BM-Tempered}
Let the random variable $N \sim \mathcal{N}(0,1)$,   the assumptions and notations of Theorem \ref{thm:Breuer-Major_tempered} hold, and $\sigma^{2}_{J,H,\lambda,d}>0$.
 If $f \in L^{2}(\R,\gamma)$ with $\E_\gamma[f]=0$ and  belongs to Sobolev space $\mathbb{D}^{1,4}(\R,\gamma)$, then
\begin{equation}\label{eq:TV-rate-1}
d_{TV} \left(  \frac{V^{J}_n}{\sqrt{\Var\left( V^J_n  \right)}}, N  \right) = \mathcal{O} \left(
n^{-\frac{1}{2}} \left(  \sum_{\vert \nu \vert \le n} \Big \vert \gamma^J(\nu)  \Big \vert \right)^{\frac{3}{2}}\right),\; n\rightarrow \infty.
\end{equation}
 So, $d_{TV} \left(  \frac{V^J_n}{\sqrt{\Var\left( V^J_n  \right)}}, N  \right) \le C \, n^{-\frac{1}{2}}$ for some constant $C$, and $n \ge 1$. Here $J=I, II$.
\end{thm}

\begin{proof}
 Both estimates \eqref{eq:TV-rate-1} for $J=I, I\!I$ are direct consequence of \cite[Theorem 1.2]{n-p-y19} (recalling as Theorem \ref{thm:BM-TV-Rate} in Appendix  B) along with the fact that the limiting variances given by relation   \eqref{eq:Target-Variance-TemperedI}  are non zero by our assumption and bounded. Moreover, those estimate can be further controlled from above relying on the fact that $\gamma^I, \gamma^{I\!I} \in l^1(\Z)$ by Proposition \ref{prop:TFLN}.
\end{proof}

\begin{rem}  \begin{itemize}
\item[$(a)$] Clearly, Theorem \ref{thm:quantitative-TV-BM-Tempered} implies Theorem \ref{thm:Breuer-Major_tempered} under the extra assumptions that the  function  $f \in \mathbb{D}^{1,4}(\R,\gamma)$. Up to our knowledge, this is the minimal assumption required to obtain a rate of convergence in the total variation metric. It is clear that without imposing such regularity assumption any reasonable rate of convergence in the total variation distance is implausible.
	\item[$(b)$]	  In general, it is an open problem in the field to provide the similar lower rate of the convergence, namely, to establish that  for some positive constant $C>0$
		\begin{equation*}
	  d_{TV} \left(  \frac{V^{J}_n}{\sqrt{\Var\left( V^{J}_n  \right)}}, N  \right)  \ge 	C \,  n^{-\frac{1}{2}}, J=I,II.
		\end{equation*}
A partial decisive answer is given by the so called optimal fourth moment theorem \cite{n-p-optimal}, recalling as Theorem \ref{thm:4MT} item $(b)$ when the  function $f=H_q$ is a Hermite polynomial of degree $q\ge2$.
	\end{itemize}
\end{rem}

 Denote $F^I_n:= \frac{V^I_n}{\sqrt{\Var\left( V^I_n  \right) }}$ and  let $p^{I,(m)}_n$ and $p^{(m)}_N$ be the $m$th derivatives of densities of random variables $F_n$ and $N$ respectively, where, as before, $N \sim \mathcal{N}(0,1)$.

\begin{thm}[Density Convergence in the Breuer--Major Theorem]\label{thm:BM-Tempered-Densities}
Let    all the assumptions and notations of Theorem \ref{thm:Breuer-Major_tempered} hold, and function $f$ be given by
\[ f (x) = \sum_{j=d}^{q} a_j H_j (x),  \] where $2 \le d \le q$, and $(a_j : j=d,\ldots,q)$ are real numbers.   
Then 
\begin{description}
\item[$(a)$] For all $m\ge 0$  and every $p \in [1,\infty]$ (including $p=\infty$ corresponding to the uniform norm),
\begin{equation}\label{eq:Density-Convergence}
\Big \Vert p^{I,(m)}_n - p^{(m)}_N  \Big \Vert_{L^{p}(\R)} \longrightarrow 0
\end{equation}
  as $n$ tends to infinity.

\item[$(b)$] In particular, if $q=d$ (in other words the sequence $F_n$ belongs to the fixed Wiener chaos of order $d$), then for all $m \ge 0$  there exist  $n_0 \in \N$  and a constant $C$ (depending only on $m$ and $q$) such that for all $n \ge n_0$  we have

\begin{equation}\label{eq:Quantitative-Density-Convergence}
\Big \Vert p^{I,(m)}_n - p^{(m)}_N  \Big \Vert_{L^{\infty}(\R)}  \le C \, \sqrt{        \E   \left[F^4_n \right] -     3       }.
\end{equation}

\end{description}
Similar statements are also valid by replacing $V^I_n$ with $V^{I\!I}_n$.
\end{thm}
\begin{rem} With the particular case $p=1$, and $m=0$ in $(a)$, the above estimate implies that $d_{TV}(F_n,N) \to 0$, however it does not provide any rate of the convergence. Moreover, the assumption on the function $f$ here is  more stronger that the previous theorem. This is somehow clear due to requiring a more stronger convergence.
\end{rem}
\begin{proof}
Let $h^I$ denote the spectral density function of TFGNI. Note that  $h^I \in L^1([-\pi,\pi]) $   in virtue of   \cite[Proposition 7.3.3]{Nourdin-Peccati-Bible}. In fact, $h^I \in L^\infty ([-\pi,\pi])$, and hence $h^I \in  L^r ([-\pi,\pi])$ for every $r \ge 1$,  because $\gamma^I \in l^1 (\Z)$.  Moreover,
\[  h^I (\omega )   \approx \frac{\omega^2}{(\lambda^2 + \omega^2)}, \quad \text{ as } \quad \vert \omega \vert \to 0. \]
Hence, $\log(h^I) \in L^1([-\pi,\pi])$. Now part $(a)$ follows directly from  \cite[Theorem 1.5]{HU2} and \cite[Corollary 1.6]{HU2}, recalling as Theorem \ref{thm:BM-Density-Rate}. Proof for the case TFGNII is similar.
\end{proof}
\begin{rem}
 Condition $\log(h^I) \in L^1([  \pi,\pi])$ is referred to as purely nondeterministic property in the literature, and in particular implies that the following useful representation takes places $$ X^I(k) =  \sum_{j\ge 0} a_j \varepsilon_{k-j}$$ where $(\varepsilon_k)$ stands for  a standard white noise. Roughly speaking, Malliavin calculus bridges  density and its derivatives to existence of the negative moments on the norm of the Malliavin derivative. The latter condition exist only in some special cases, and the assumption $\log(h^I) \in L^1([  \pi,\pi])$ requires for the justification of the condition.
  \end{rem}

{Fix $q \ge 2$}.  Let $V_n=\frac{1}{\sqrt{n} }\sum_{k=1}^{n} H_q(Y^J(k)), J=I,I\!I$. Consider the sequence $(F^J_n : n \ge 1)$ defined via
\begin{equation}\label{eq:hermite-variation}
F^J_n = \frac{V_n}{\sqrt{\Var\left( V_n \right)}}=  \frac{1}{\sqrt{n \Var\left(  V_n\right)}} \sum_{k=1}^{n} H_q(Y^J(k)).
\end{equation}

\begin{thm}[Exact asymptotics in the Breuer--Major CLT]\label{thm:Exact-BM-General-Chaos}
Let $N \sim  \mathcal{N}(0,1)$. Consider the sequence $(F^J_n : n\ge 1)$ given by relation \eqref{eq:hermite-variation}. Then, for every $z \in \R$, as $n\rightarrow\infty$,  {the following exact asymptotic statement takes place}

\begin{equation}\label{eq:Calim}
  \sqrt{n} \Big(   \mathbb{P}\left( F^J_n \le z \right) - \mathbb{P} \left(  N \le z \right)   \Big)  \longrightarrow     \frac{ \rho}{3\sqrt{2\pi} }  (z^2 -1 ) e^{- \frac{z^2}{2}},
  \end{equation}
with
\begin{align}
\rho &= C^{-3q}_1 \, qq! (q/2)! { q-1 \choose q/2 -1 }^2 \frac{1}{\sigma^3} \sum_{k,l \in \Z} \gamma^J(k)^{q/2} \gamma^J(l)^{q/2} \gamma^J(l-k)^{q/2}\nonumber
\end{align}
and
\begin{align}
\sigma^2 &= q!  C^{-2q}_1  \sum_{k \in \Z} \left(\gamma^J(k)\right)^q >0\nonumber
\end{align}
  {provided that either $J=I,I\!I$, $q$ even, or $J=I$, $H \in (0,1/2]$, or $J=I\!I$, $H \ge 1/2$.}
\end{thm}

\begin{proof}
We only consider the case $J=I$. The other case is similar. We are going to apply \cite[Theorem 3.1]{n-p-AOP-Exact-Asymptotic}. In order to settle in that framework,  we can assume, without loss of generality, that $X^I (k)= X (\varepsilon_k)$, where $\{ X(h) : h \in \HH \}$ is an adequate isonormal Gaussian process over the separable Hilbert space $\HH$ (see \cite[Proposition 7.2.3]{Nourdin-Peccati-Bible}) with $\langle  \varepsilon_k, \varepsilon_l  \rangle_{\HH} = C^{-2}_1\gamma^I (k-l)$ for every $k,l \in \Z$. So, we can write
\[  F_n = I_q (f_n), \,  \text{ and } \,  f_n :=  \frac{1}{\sqrt{n \Var \left( V^I_n \right)}} \sum_{k=1}^{n} \varepsilon_k^{\otimes q},  \quad n \ge 1.  \]
 The notation $ \varepsilon_k^{\otimes q}$ stands for the tensor product.
First, note that as in Lemma \ref{Hermite}, dominated convergence theorem yields that
\begin{equation*}\begin{gathered}
\sigma^2_n : = \Var \left( V^I_n\right) = \frac{q!}{n} \,   C^{-2q}_1  \sum_{k,l=1}^{n} \left(\gamma^I (k-l)\right)^q  = q! \, C^{-2q}_1  \sum_{\vert k \vert < n}  \left(1 - \frac{\vert k \vert }{n}\right) \left(\gamma^I (k)\right)^q \\\longrightarrow \sigma^2:=  q!  \,  C^{-2q}_1  \sum_{k \in \Z} \left(\gamma^I (k)\right)^q < + \infty.
\end{gathered}\end{equation*}
Therefore, according to Breuer--Major theorem we can conclude that $F_n \stackrel{d}{\longrightarrow} \mathcal{ N }(0,\sigma^2)$ as $n$ tends to infinity, and that $\sigma^2 >0$ according to Lemma \ref{Hermite}. Furthermore,   $$DF_n = \frac{q}{\sqrt{n \sigma^2_n}}\sum_{k=1}^{n} \varepsilon_k I_{q-1} \left( \varepsilon^{\otimes (q-1)}_k \right).$$ Hence, using product formula for multiple integrals \eqref{multiplication}, we can write

\begin{multline*}
 \Vert DF_n \Vert^2_{\HH}= \frac{q^2}{n \sigma^2_n} \sum_{k,l=1}^{n}  C^{-2}_1 \gamma^I (k-l) I_{q-1} \left( \varepsilon^{\otimes (q-1)}_k \right) I_{q-1} \left( \varepsilon^{\otimes (q-1)}_l \right)\\
 = \frac{q^2 }{n \sigma^2_n} \sum_{k,l=1}^{n} C^{-2}_1 \gamma^I (k-l)  \sum_{r=0}^{q-1} r! { q-1 \choose r}^2 I_{2q-2r-2} \left(  \varepsilon^{\otimes (q-1)}_k  \otimes_{r}  \varepsilon^{\otimes (q-1)}_l \right)\\
 =  \frac{q^2 }{n \sigma^2_n} \sum_{k,l=1}^{n}  \sum_{r=0}^{q-1} r! { q-1 \choose r}^2  C^{-2(r+1)}_1 \gamma^I (k-l)^{r+1}    I_{2q-2r-2} \left(  \varepsilon^{\otimes (q-r-1)}_k  \otimes  \varepsilon^{\otimes (q-r-1)}_l \right)\\
 =\frac{q^2}{n \sigma^2_n} \sum_{r=1}^{q} (r-1)! { q-1 \choose r-1}^2 \sum_{k,l=1}^{n} I_{2q-2r} \left( \varepsilon^{\otimes (q-r)}_k  \otimes \varepsilon^{\otimes (q-r)}_l\right)  C^{-2r}_1 \gamma^I (k-l)^{r}.
 \end{multline*}
 Therefore,

 \begin{multline*}
G_n := \frac{1}{q} \Vert DF_n \Vert^2_{\HH} -1= \frac{q}{n \sigma^2_n} \sum_{r=1}^{q} (r-1)! { q-1 \choose r-1}^2 \sum_{k,l=1}^{n} I_{2q-2r} \left( \varepsilon^{\otimes (q-r)}_k  \otimes \varepsilon^{\otimes (q-r)}_l\right) C^{-2r}_1 \gamma^I (k-l)^{r}-1\\
=\frac{q}{n \sigma^2_n} \sum_{r=1}^{q-1} (r-1)! { q-1 \choose r-1}^2 \sum_{k,l=1}^{n} I_{2q-2r} \left( \varepsilon^{\otimes (q-r)}_k  \otimes \varepsilon^{\otimes (q-r)}_l\right)  C^{-2r}_1 \gamma^I (k-l)^{r}.
 \end{multline*}

 One has to note that for each $n\ge 1$, the random variable $G_n$ belongs to a finite sum of Wiener chaoses up to order $2q-2$. Our next aim is to show that $\sqrt{n}G_n \to \mathcal{N}(0,\widehat{\sigma}^2)$ as $n\rightarrow\infty$   for some variance $\widehat{\sigma}^2$ whose value will be determined later on. To do this, we apply the  fourth moment Theorem \ref{thm:4MT}. For each $r \in \{ 1,\ldots,q-1  \}$, set $$\varrho=\varrho(q,r,H,\lambda):= q (r-1)!{q-1 \choose r-1}^2$$ and  define
 \begin{equation}\label{eq:DF_n-r}
 G_{n,r} : =   \frac{\varrho}{\sigma^2_n \sqrt{n}} \sum_{k,l=1}^{n} I_{2q-2r} \left( \varepsilon^{\otimes (q-r)}_k  \otimes \varepsilon^{\otimes (q-r)}_l\right)  C^{-2r}_1  \gamma^I (k-l)^{r}.
 \end{equation}
 First, using Proposition \ref{prop:TFLN} part $(a)$,  we obtain  that

  \begin{multline}\label{eq:r-Variance}
  \sigma^2_{n,r} : = \Var (G_{n,r}) =  (2q-2r)! \frac{\varrho^2}{n \, \sigma^4_n}   \sum_{k_1,l_1,k_2,l_2=1}^{n}  C^{-2r}_1  \gamma^I (k_1 -l_1)^r  C^{-2r}_1  \gamma^I (k_2 - l_2)^r \\ \hskip6cm   \times  C^{-2(q-r)}_1 \gamma^I(k_1 - k_2)^{q-r} C^{-2(q-r)}_1 \gamma^{I}(l_1 -l_2 )^{q-r}\\
  \to \frac{(2q-2r)!\varrho^2}{\sigma^4}  C^{-4q}_1 \sum_{ k_1,k_2,k_3\in \Z}  \gamma^I (k_1)^r \gamma^I(k_2)^{q-r}\gamma^I (k_3)^{r} \gamma^I (k_2 + k_3 - k_1)^{q-r}=: \sigma^2_r < +\infty
  \end{multline}
     as $n\rightarrow\infty$.
  Next, we will show that for each $r \in \{ 1,\ldots,q-1  \}$,  we have that
  \begin{equation}\label{eq:G^r-CLT}
  \widetilde{G}_{n,r} : = \frac{G_{n,r}}{\sqrt{\sigma^2_{n,r}}} \limd   \mathcal{N}(0,1)
  \end{equation}
   as $n\rightarrow\infty$.
To start with, note that
  \begin{equation*}
  D\widetilde{G}_{n,r} = \frac{(2q-2r)\varrho}{\sigma_{n,r} \sigma^2_n} \times  \frac{1}{\sqrt{n}}  \sum_{k,l=1}^{n} \varepsilon_k I_{2q-2r-1} \left(  \varepsilon^{\otimes (q-r-1)}_k  \otimes \varepsilon^{\otimes (q-r)}_l\right)  C^{-2r}_1 \gamma^I (k-l)^{r},
  \end{equation*}
  Therefore,
  \begin{multline*}
  \Vert D\widetilde{G}_{n,r} \Vert^2_{\HH} = \left(  \frac{(2q-2r) \varrho}{\sigma_{n,r} \sigma^2_n} \right)^2\\
  \times \Bigg[  \frac{1}{n} \sum_{k_1,l_1,k_2,l_2=1}^{n}  C^{-4r}_1 \gamma^I (k_1 - l_1)^r \gamma^I (k_2 - l_2)^r  C^{-2}_1  \gamma^I (k_1 - k_2) \\
  \times I_{2q-2r-1} \left(  \varepsilon^{\otimes (q-r-1)}_{k_1} \otimes   \varepsilon^{\otimes (q-r)}_{l_1}\right)  I_{2q-2r-1} \left(  \varepsilon^{\otimes (q-r-1)}_{k_2}
  \otimes \varepsilon^{\otimes (q-r)}_{l_2}\right)   \Bigg]  \\
  = \left(  \frac{(2q-2r)\varrho}{\sigma_{n,r} \sigma^2_n} \right)^2\\
  \times \Bigg[  \frac{1}{n} \sum_{k_1,l_1,k_2,l_2=1}^{n}  C^{-4r}_1 \gamma^I (k_1 - l_1)^r \gamma^I (k_2 - l_2)^r  C^{-2}_1   \gamma^I (k_1 - k_2) \\
  \sum_{s=0}^{2q-2r-1} s! {  2q-2r-1 \choose s}^2 I_{4q-4r-2-2s}
   \left(  \left(  \varepsilon^{\otimes (q-r-1)}_{k_1} \otimes  \varepsilon^{\otimes (q-r)}_{l_1}\right) \otimes_s   \left(  \varepsilon^{\otimes (q-r-1)}_{k_2}  \otimes \varepsilon^{\otimes (q-r)}_{l_2}\right)  \right)  \Bigg],  \\
  \end{multline*}
 and consequently,
 \begin{multline*}
 \frac{1}{(2q-2r)} \Vert D \widetilde{G}_{n,r} \Vert^2_{\HH} -1 = \frac{(2q-2r)\varrho^2}{\sigma^2_{n,r} \sigma^4_n} \times \Bigg[  \frac{1}{n} \sum_{k_1,l_1,k_2,l_2=1}^{n}  C^{-4r}_1 \gamma^I (k_1 - l_1)^r \gamma^I (k_2 - l_2)^r  C^{-2}_1  \gamma^I (k_1 - k_2) \\
 \sum_{s=0}^{2q-2r-2} s! {  2q-2r-1 \choose s }^2 I_{4q-4r-2-2s}
 \left(  \left(  \varepsilon^{\otimes (q-r-1)}_{k_1} \otimes  \varepsilon^{\otimes (q-r)}_{l_1}\right) \otimes_s   \left(  \varepsilon^{\otimes (q-r-1)}_{k_2}  \otimes \varepsilon^{\otimes (q-r)}_{l_2}\right)  \right)  \Bigg]  \\
 = \frac{(2q-2r)\varrho^2}{\sigma^2_{n,r} \sigma^4_n} \times \sum_{s=0}^{2q-2r-2} s! {  2q-2r-1 \choose s}^2\\
 \times \Bigg[   \frac{1}{n} \sum_{k_1,l_1,k_2,l_2=1}^{n}   C^{-4r}_1 \gamma^I (k_1 - l_1)^r \gamma^I (k_2 - l_2)^r C^{-2}_1  \gamma^I (k_1 - k_2) \\ \times I_{4q-4r-2-2s}
 \left(  \left(  \varepsilon^{\otimes (q-r-1)}_{k_1} \otimes  \varepsilon^{\otimes (q-r)}_{l_1}\right) \otimes_s   \left(  \varepsilon^{\otimes (q-r-1)}_{k_2}  \otimes \varepsilon^{\otimes (q-r)}_{l_2}\right)  \right) \Bigg]\\
 = \frac{(2q-2r)\varrho^2}{\sigma^2_{n,r} \sigma^4_n}\times \sum_{s=0}^{q-r} s! {  2q-2r-1 \choose s}^2\\
 \times \Bigg[   \frac{1}{n} \sum_{k_1,l_1,k_2,l_2=1}^{n}  C^{-4r}_1   \gamma^I (k_1 - l_1)^r \gamma^I (k_2 - l_2)^r C^{-2}_1 \gamma^I (k_1 - k_2) \\
 \times I_{4q-4r-2-2s} \left(  \left(  \varepsilon^{\otimes (q-r-1)}_{k_1} \otimes  \varepsilon^{\otimes (q-r-1)}_{k_2} \right) \otimes \left(  \varepsilon^{\otimes (q-r-s)}_{l_1} \otimes  \varepsilon^{\otimes (q-r-s)}_{l_2}  \right)  \right)  C^{-2s}_1 \gamma^I (l_1 -l_2)^{s} \Bigg]\\
 + \frac{(2q-2r)\varrho^2}{\sigma^2_{n,r} \sigma^4_n} \times \sum_{s=q-r+1}^{2q-2r-2} s! {  2q-2r-1 \choose s}^2\\
 \times \Bigg[   \frac{1}{n} \sum_{k_1,l_1,k_2,l_2=1}^{n}  C^{-4r}_1  \gamma^I (k_1 - l_1)^r \gamma^I (k_2 - l_2)^r C^{-2}_1 \gamma^I (k_1 - k_2) \\
 \times I_{4q-4r-2-2s} \left(   \varepsilon^{\otimes (2q-2r-1-s)}_{k_1} \otimes  \varepsilon^{\otimes (2q-2r-1-s)}_{k_2} \right) C^{-2(s+1)}_1   \gamma^I (l_1 -l_2)^{q-r} \gamma^I(k_1 -k_2)^{s+1-q+r} \Bigg]\\
 = \frac{(2q-2r)\varrho^2}{\sigma^2_{n,r} \sigma^4_n}\times
 \sum_{s=1}^{q-r+1} (s-1)! {  2q-2r-1 \choose s-1}^2\\
 \times \Bigg[   \frac{1}{n} \sum_{k_1,l_1,k_2,l_2=1}^{n} C^{-4r}_1 \gamma^I (k_1 - l_1)^r \gamma^I (k_2 - l_2)^r C^{-2}_1 \gamma^I (k_1 - k_2) \\
 \times I_{4q-4r-2s} \left(  \left(  \varepsilon^{\otimes (q-r-1)}_{k_1} \otimes  \varepsilon^{\otimes (q-r-1)}_{k_2} \right) \otimes \left(  \varepsilon^{\otimes (q-r-s+1)}_{l_1} \otimes  \varepsilon^{\otimes (q-r-s+1)}_{l_2}  \right)  \right) C^{-2(s-1)}_1   \gamma^I (l_1 -l_2)^{s-1} \Bigg]\\
  + \frac{(2q-2r)\varrho^2}{\sigma^2_{n,r} \sigma^4_n} \times \sum_{s=q-r+2}^{2q-2r-1} (s-1)! {  2q-2r-1 \choose s-1}^2\\
 \times \Bigg[   \frac{1}{n} \sum_{k_1,l_1,k_2,l_2=1}^{n}  C^{-4r}_1  \gamma^I (k_1 - l_1)^r \gamma^I (k_2 - l_2)^r  C^{-2}_1 \gamma^I (k_1 - k_2) \\
 \times I_{4q-4r-2s} \left(   \varepsilon^{\otimes (2q-2r-s)}_{k_1} \otimes  \varepsilon^{\otimes (2q-2r-s)}_{k_2} \right) C^{-2s}_1   \gamma^I (l_1 -l_2)^{q-r} \gamma^I(k_1 -k_2)^{s-q+r} \Bigg]\\
 \end{multline*}
 Now, for every $1 \le s \le q-r+1$, we have

 \begin{multline*}
 \E \Bigg \vert        \frac{1}{n} \sum_{k_1,l_1,k_2,l_2=1}^{n}  C^{-4r}_1  \gamma^I (k_1 - l_1)^r \gamma^I (k_2 - l_2)^r C^{-2}_1 \gamma^I (k_1 - k_2) \\
 \times I_{4q-4r-2s} \left(  \left(  \varepsilon^{\otimes (q-r-1)}_{k_1} \otimes  \varepsilon^{\otimes (q-r-1)}_{k_2} \right) \otimes \left(  \varepsilon^{\otimes (q-r-s+1)}_{l_1} \otimes  \varepsilon^{\otimes (q-r-s+1)}_{l_2}  \right)  \right) C^{-2(s-1)}_1  \gamma^I (l_1 -l_2)^{s-1}             \Bigg \vert^2\\
 =\frac{1}{n^2} \sum_{k_1,l_1,k_2,l_2,k_3,l_3,k_4,l_4=1}^{n} \Bigg[ C^{-8r-4s}_1  \gamma^I(k_1 -l_1)^{r} \gamma^I(k_2 -l_2)^r \gamma^I(k_1 - k_2) \gamma^I(l_1 - l_2)^{s-1}\\
 \times \gamma^I (k_3 - l_3)^r \gamma^I(k_4 - l_4)^r \gamma^I(k_3 - k_4) \gamma^I(l_3 -l_4)^{s-1}\\
 \times  C^{-2(4q-4r-2s)}_1 \gamma^I(k_1 -k_3)^{q-r-1} \gamma^I (k_2 - k_4)^{q-r-1}\gamma^I(l_1 - l_3)^{q-r-s+1} \gamma^I(l_2 - l_4)^{q-r-s+1}      \Bigg] \\
 \sim_{n \to +\infty}  C^{-8q}_1  \frac{1}{n} \sum_{x_1,...,x_7 \in \Z} \Bigg[    \gamma^I(x_1)^r \gamma^I(x_2)^r \gamma^I(x_3)\gamma^I(x_2+x_3 - x_1)^{s-1}\\
 \times \gamma^I(x_4)^r \gamma^I(x_5)^r \gamma^I(x_6)\gamma^I(x_5+x_6 -x_4)^{s-1} \gamma^I(x_7)^{q-r-1}\\
 \times \gamma^I(x_6+x_7-x_3)^{q-r-1} \gamma^I(x_4 +x_7 -x_1)^{q-r-s+1} \gamma^I(x_5+x_6+x_7-x_2-x_3)^{q-r-s+1}\Bigg] \\
 \to 0, \, \text{ as }  \, n \to \infty.
 \end{multline*}
 Similarly for each $s \in \{ q-r+2, ..., 2q-2r-1  \}$, one can show that
 \begin{multline*}
 \E \Bigg  \vert         \frac{1}{n} \sum_{k_1,l_1,k_2,l_2=1}^{n} \gamma^I (k_1 - l_1)^r \gamma^I (k_2 - l_2)^r \gamma^I (k_1 - k_2) \\
 \times I_{4q-4r-2-2s} \left(   \varepsilon^{\otimes (2q-2r-1-s)}_{k_1} \otimes  \varepsilon^{\otimes (2q-2r-1-s)}_{k_2} \right)  \gamma^I (l_1 -l_2)^{q-r} \gamma^I(k_1 -k_2)^{s-q+r}     \Bigg \vert^2\\
 \to 0\, \text{ as }  \, n \to \infty.
 \end{multline*}
 Hence, using the orthogonality property of multiple stochastic integrals, one can infer that
 \[  \Var  \left(  \frac{1}{2q-2r} \Vert D \widetilde{G}_{n,r} \Vert^2_{\HH} -1   \right) \to 0, \] and the latter immediately implies \eqref{eq:G^r-CLT}. Furthermore, taking into account
 \[ \sqrt{n}   \left(  \frac{1}{q}  \Vert DF_n \Vert^2_{\HH} -1 \right) = \sum_{r=1}^{q-1}  G_{n,r} \]
 and, using Peccati--Tudor multidimensional fourth moment Theorem \ref{thm:4MT-MultiDim}, we can infer that, as $n$ tends to infinity, we have

 \[   \sqrt{n}   \left(  \frac{1}{q}  \Vert DF_n \Vert^2_{\HH} -1 \right) \limd \mathcal{N}(0,\widehat{\sigma}^2), \]
 with
 \begin{equation}\label{eq:Final-Variance}
 \widehat{\sigma}^2 = \sum_{r=1}^{q-1}  \sigma^2_r,
  \end{equation}
  where $\sigma^2_r$ is given by relation \eqref{eq:r-Variance}. Therefore,  \cite[Theorem 2.6]{n-p-AOP-Exact-Asymptotic} part (B) (recalling as Theorem \ref{thm:4MT},   part {\bf (c)}) yields that
  \[  \left( F_n,   \sqrt{n}   \left(  \frac{1}{q}  \Vert DF_n \Vert^2_{\HH} -1 \right) \right) \limd (N_1,N_2),   \]
  where $(N_1,N_2)$ is a centered two dimensional Gaussian vector with $\E[N^2_1] =1$, $\E[N^2_2] = \widehat{\sigma}^2$, and $\E[N_1 \times N_2]= \rho$, where, by orthogonality of multiple stochastic integrals,

  \begin{align*}
  \rho   & = \lim_{n\to \infty}  \E \left[ F_n \times   \sqrt{n}   \left(  \frac{1}{q}  \Vert DF_n \Vert^2_{\HH} -1 \right)  \right] =
   \lim_{n \to \infty}  \E \left[ F_n \times G_{n,q/2} \right]\\
  &=   q!  q (q/2)! { q-1 \choose q/2 -1 }^2  \lim_{n\to \infty} \frac{1}{\sigma^3_n} \times \frac{1}{n} \sum_{k,l,t=1}^{n} C^{-q}_1  \gamma^I(l-t)^{q/2}  C^{-q}_1  \gamma^I(l-k)^{q/2}  C^{-q}_1  \gamma^I(t-k)^{q/2}\\
  & \longrightarrow  C^{-3q}_1  q!  q  (q/2)! { q-1 \choose q/2 -1 }^2 \frac{1}{\sigma^3} \sum_{k,l\in \Z} \gamma^I(k)^{q/2} \gamma^I(l)^{q/2} \gamma^I(l-k)^{q/2}.
  \end{align*}
Finally we can deduce the claim \eqref{eq:Calim}.
\end{proof}

\begin{rem}[Exact asymptotics in the Breuer--Major CLT on the second Wiener chaos]\label{rem:Exact-BM-2Chaos}
Let $N \sim \mathcal{N}(0,1)$. Let $Y^J(k)= C^{-1}_1 \left(X^J(k+1)-X^J(k) \right)$ where $X^J(k)$ is either TFBM  (J=I) or TFBMII (J=$I\!I$) with the associated normalizing constant $C_1$ appearing in Lemma \ref{lem:usefulrelations}, covariance function $\gamma^J$ and spectral density function $h^J$. Let $ {q = 2}$. Consider the sequence $(F^J_n : n\ge 2)$ given by
relation \eqref{eq:hermite-variation} belonging to the second Wiener chaos. In this case, thanks to the fact that $\gamma^J \in l^{4/3} (\Z)$ (by Proposition \ref{prop:TFLN}), one can apply \cite[Proposition 3.8]{n-p-AOP-Exact-Asymptotic} (or \cite[Theorem 9.5.1]{Nourdin-Peccati-Bible}) to readily obtain, for every $z \in \R$, as $n$ tends to infinity, that
\[  \sqrt{n} \Big( \mathbb{P} \left( F^J_n \le z   \right) - \mathbb{P} \left(   N \le z\right) \Big)   \longrightarrow  \frac{\Vert h^J \Vert^3_{L^3([-\pi,\pi])}}{6 \pi \sqrt{\pi} \Vert \gamma^J \Vert^3_{l^2(\Z)}} (z^2 -1)e^{-\frac{z^2}{2}}. \]
\end{rem}

\begin{thm}[Almost Sure Convergence in the Breuer--Major  CLT]\label{thm:Almost-Sure-CLT}
	Let $N \sim \mathcal{N}(0,1)$, and $f = \sum_{q=1}^{\infty} a_q H_q (x) \in L^2(\R,\gamma)$ where as before $\gamma$ stands for the standard Gaussian measure on the real line.  For every $n\ge 1$ define $V^J_n  :=   \frac{1}{\sqrt{n}}  \sum_{k=1}^{n}  f(X^J_k)$. Consider sequence $F^J_n:= \frac{V^J_n}{\sqrt{\Var(V^J_n)}}$.  
	If in addition, function $f$ is of the class $C^2(\R)$ such that $\E[f''(N)^4] < \infty$, then the sequence $(F^J_n : n \ge 1)$ satisfies an ASCLT meaning that, almost surely, for every bounded continuous function $\varphi : \R \to \R$ it holds that
	
	\[  \frac{1}{\log n} \sum_{k=1}^{n} \frac{1}{k} \varphi (F^J_n)  \longrightarrow \E \left[  \varphi(N) \right], \quad \text{ as } \quad n \to \infty, \]
	provided that either $J=I,I\!I$ and there is at least one even $q$ so that $a_q \neq 0$, or $J=I$, $H \in (0,1/2]$, or $J=I\!I$, $H \ge 1/2$ and there is at least one coefficient $a_q \neq 0$ for $q>1$ in both latter cases.
	\end{thm}

\begin{proof}
Note that $\gamma^J \in l^1(\Z)$ for $J=I,I\!I$ by Proposition \ref{prop:TFLN}. Now, the claim can be straightforwardly achieved using Theorem 3.4, and Remark 3.5 in \cite{b-n-t-almost-sure}.
\end{proof}

\begin{rem}\label{rem:ASCLT:Hermite-Variation}
	The tempering parameter $\lambda$ removes completely the presence of any extra restriction on the Hurst parameter $H$ in the ASCLT for the $q$th Hermite variation of tempered fractional Gaussian noises, at least when $q$ is even. For the classical fractional Gaussian noise, we refer the reader to \cite[Theorem 6.2]{b-n-t-almost-sure}.
	\end{rem}

We next investigate the asymptotic behavior  of the third and fourth cumulants  of tempered fractional Gaussian processes. First, we define the cumulants of a random variable.
Let $F$ be a real-valued random variable with $\mathbb{E}|F|^n <\infty$ for $n\geq 1$.  Let $\phi_{F}(t)=\mathbb{E}[e^{itF}]$ be the characteristic function of $F$.
Then
$$
\kappa_{j}(F) = (-i)^j \frac{d^j}{dt^j}\log \phi_{F}(t)|_{t=0}
$$
is called the $j$th   cumulant of  F. 
For every $n \ge 1$, recall that
\begin{equation}\label{eq:Hermite-Variation}
F^J_n = \frac{V^J_n}{\sqrt{\Var\left( V^J_n \right)}}=  \frac{1}{\sqrt{n \Var\left(  V^J_n\right)}} \sum_{k=1}^{n} H_q(Y^J(k)), \quad J=I,I\!I.
\end{equation}

\begin{prop}[Optimal 3rd moment theorem]\label{optimal_fourth_moment}
	Let $q \ge 2$ be an integer. Consider sequence $(F^J_n \, : \, n \ge 1)$ given by relation \eqref{eq:Hermite-Variation} where $H_q$ denote the Hermite polynomial of degree $q$.	Then, as $n$ tends to infinity,
	\begin{itemize}
		\item[(a)]  For any even integer $q\geq 2$, it holds that $\kappa_{3}(F^J_n)\asymp n^{-\frac{1}{2}}$.
		
		\item[(b)] For any integer $q\geq 2$, it holds that $\kappa_{4}(F^J_n)\asymp n^{-1}$ provided that either $q$ is even, or $J=I$, $H \in (0,1/2]$, or $J=I\!I$, $H \ge 1/2$.
	\end{itemize}
		Therefore, if $q\ge 2$ is an  even  integer, then there exist two constants $C_1, C_2 >0$ (independent of $n$) so that for every $n \ge1$, the following optimal third moment estimate holds:
		
		\begin{equation}\label{eq:Optimal-Third-Moment}
		C_2  \, \big \vert \E [(F^J_n)^3]  \big\vert \le  d_{TV} (F^J_n, N) \le C_1  \,   \big \vert \E [(F^J_n)^3]  \big\vert .
		\end{equation}
	
\end{prop}

\begin{proof}
	(a) First, Proposition \ref{prop:TFLN} implies that  $\gamma^J \in \ell^{\frac{3q}{4}} (\Z)$. Now, from \cite[Proposition 6.3]{Bierme} we obtain that $0<\liminf \sqrt{n}\kappa_{3}(F_n)= \limsup \sqrt{n}\kappa_{3}(F_n)<\infty$ and this gives the desired result in part $(a)$.  For part (b) using Proposition \ref{prop:TFLN} infer that $\gamma^J \in l^{2}(\mathbb{Z})$, and hence \cite[Proposition 6.4]{Bierme} completes the proof in virtue of Lemma \ref{Hermite}. Finally, relation \eqref{eq:Optimal-Third-Moment} is a direct application of \cite[Theorem 2.1]{n-p-optimal} (recalling as Theorem \ref{thm:4MT} part {\bf (b)} in the appendix section).
\end{proof}

\begin{rem} \label{rem:optimalrate}
	\begin{itemize}
	\item[(i)]	See \cite[Remark 8.4.5]{Nourdin-Peccati-Bible}, item $1$ when $q$ is odd. In fact,  for a random variable $F$ belonging to a fixed Wiener chaos of odd order all the odd cumulants vanish. On the other hand, for a given general random variable $F$ with $\E[F]=0$, and $\E[F^2]=1$ we have $\kappa_3 (F)=\E[F^3]$ and $\kappa_4(F)=\E[F^4]-3$.  
	Hence, when $q$ is odd, then as explained $\E[(F^J_n)^3]=0$, and therefore the  optimal rate of convergence in the total variation metric is given by $\kappa_{4}(F^J_n)$ that is equals to $n^{-1}$ under the extra assumption that either $J=I$, $H \in (0,1/2]$ or $J=I\!I$, $H\ge 1/2$.   When $q$ is even, then there is a fight between the third and fourth cumulants in the optimal rate.   Proposition  \ref{optimal_fourth_moment} states  that in this situation the third cumulant $\kappa_3(F^J_n)$ is the winner, and the rate of convergence is $n^{-1/2}$.

\item[(ii)] The tempering parameter $\lambda$ manifests its role in the optimal fourth moment theorem. In fact, the optimal rates of convergence of the third and fourth cumulants of $F^J_n$ given by Proposition \ref{optimal_fourth_moment} are valid for any $H>0$ and $\lambda>0$ for even $q$. This is in contrast with the case of fractional Brownian motion where $\kappa_{3}(F_n)\asymp n^{-\frac{1}{2}}$ provided $H\in (0, 1-\frac{2}{3q})$ with an even integer $q\geq 2$ and $\kappa_{4}(F_n)\asymp n^{-1}$ provided $H\in (0, 1-\frac{3}{4q})$ with $q\in {2,3}$, see Propositions 6.6 and 6.7 in \cite{Bierme}. It is also worth to mention that for $q$ even the sequence $(F^J_n : n\ge 1)$ given by \eqref{eq:Hermite-Variation} exhibits the interesting scenario that $\kappa_{3}(F^J_n)  \approx \left( \kappa_{4}(F^J_n) \right)^{\frac{1}{2}}$, and hence the third cumulant $\kappa_{3}(F^J_n)$ asymptotically dominates the fourth cumulant as $n$ tends to infinity. A similar phenomenon appears in \cite{v-3Cumulant} as well, in which, convergence of third cumulants to zero implies the convergence of the fourth cumulants to zero.
    \end{itemize}
\end{rem}

\section{Acknowledgments}
Farzad Sabzikar would like to thank David Nualart for stimulating discussion on the proof of Theorem \ref{theo:variation} as well as suggesting to investigate the role of tempering in the optimal fourth moment theorem \cite{n-p-optimal}. Yu. Mishura was partially supported by the ToppForsk project nr. 274410 of the Research Council of Norway with title STORM: Stochastics for Time-Space Risk Models.

\section{Appendix  A }\label{sec:appendix}
This appendix contains some notations, definitions and well known results that we applied in the main text of this paper.

\subsection{Special functions $K_{\nu}$ and ${_2F_3}$}

In this subsection we present definitions of two special functions $K_{\nu}$ and ${_2F_3}$ that we have used in Section \ref{sec2}. We also provide the proof of Lemma \ref{lem:propo24}, see below, that we used in the proof of Proposition \ref{lem:usefulrelations 1}.
First, we start with the definition of the modified Bessel function of the second kind that appears in the variance and covariance function of TFBM, see part $(a)$ of Lemma \ref{lem:usefulrelations}. A modified Bessel function of the second kind $K_{\nu}(x)$ has the integral representation
\begin{equation*}
K_{\nu}(x)=\int_0^\infty e^{-x \cosh t} \cosh {\nu t}\ dt,
\end{equation*}
where $\nu>0, x>0$. The function $K_{\nu}(x)$ also has the series representation
\begin{equation*}
K_{\nu}(x) = \frac{1}{2}\pi \frac{ I_{-\nu}(x) - I_{\nu}(x) }{\sin(\pi \nu)},
\end{equation*}
where $I_{\nu}(x)=(\frac{1}{2}|x|)^{\nu} \sum_{n=0}^{\infty} \frac{ ( \frac{1}{2}x)^{2n}  }{n! \Gamma(n+1+\nu)}$ is called the Bessel function. We refer the reader to see (\cite[Section 8.43]{Gradshteyn}, pages 140--1414) for more information about the modified Bessel function of the second kind.

Next, we define the confluent Hypergeometric function ${_2F_3}$ that we used to obtain the variance and covariance of TFBMII, see part $(b)$ of Lemma \ref{lem:usefulrelations}. In general, a generalized hypergeometric function ${_pF_q}$ is defined by
\begin{equation*}
{_pF_q}(a_1,\cdots, a_p, b_1, \cdots, b_q, z)= \sum_{k=0}^{\infty} \frac{ (a_1)_{k} (a_2)_{k} \cdots (a_p)_{k}   }{ (b_1)_{k} (b_2)_{k} \cdots (b_q)_{k}  }\frac{z^k}{k!},
\end{equation*}
where $(c_i)_{k}=\frac{\Gamma(c_i + k)}{\Gamma(k)}$ is called Pochhammer Symbol. Therefore
\begin{equation*}
{_2F_3}( \{a_1, a_2\}, \{b_1, b_2, b_3\}, z)= {_2F_3}(a_1, a_2, b_1, b_2, b_3, z)= \sum_{k=0}^{\infty} \frac{ \Gamma(a_1 + k)\Gamma(a_2 + k) \Gamma(k) }{ \Gamma(b_1 + k)\Gamma(b_2 + k) \Gamma(b_3 + k)  } \frac{z^k}{k!},
\end{equation*}
\begin{lem}\label{lem:propo24}
Integral $I=\int_{0}^{\infty}\left(\int_{0}^{\infty}(s+x)^{H-3/2}e^{-\lambda(s+x)}ds\right)^2dx$ is finite for any $H>0$.
\end{lem}
\begin{proof} Let $H<1/2$. Then
\begin{equation*}\begin{gathered}I=\int_{0}^{\infty}\left(\int_{x}^{\infty}s^{H-3/2}e^{-\lambda s}ds\right)^2dx\leq \int_{0}^{\infty}e^{-2\lambda x}\left(\int_{x}^{\infty}s^{H-3/2}ds\right)^2dx\\=\left(H-1/2\right)^{-2}\int_{0}^{\infty}e^{-2\lambda x}x^{2H-1}dx
=\left(H-1/2\right)^{-2}(2\lambda)^{-2H}\Gamma(2H)<\infty.
  \end{gathered}\end{equation*}
  Let $H>1/2$. Then
  \begin{equation*}\begin{gathered}I=\int_{0}^{\infty}\left(\int_{x}^{\infty}s^{H-3/2}e^{-\lambda s}ds\right)^2dx\leq
  \int_{0}^{\infty}e^{-\lambda x}\left(\int_{x}^{\infty}s^{H-3/2}e^{-\frac{ \lambda s}{2}}ds\right)^2dx\\\leq
  \int_{0}^{\infty}e^{-\lambda x}\left(\int_{0}^{\infty}s^{H-3/2}e^{-\frac{ \lambda s}{2}}ds\right)^2dx
  =2^{2H-1}\lambda^{-2H}\Gamma^2(H-1/2)<\infty.
   \end{gathered}\end{equation*}
   Finally, let $H=1/2$. Then
   \begin{equation*}\begin{gathered}I=\int_{0}^{\infty}\left(\int_{x}^{\infty}s^{-1}e^{-\lambda s}ds\right)^2dx\leq
  \int_{0}^{1}x^{-1/2}\left(\int_{0}^{\infty}s^{ -3/4}e^{-\frac{ \lambda s}{2}}ds\right)^2dx\\+\int_{1}^{\infty}x^{-2}\left(\int_{0}^{\infty} e^{-\lambda s}ds\right)^2dx<\infty,
   \end{gathered}\end{equation*}
   and the proof follows.
  \end{proof}

\section{Appendix  B }\label{sec:appendix-GA}
This appendix section is devoted to the essential elements of Gaussian analysis and Malliavin calculus. For the sake of completeness, we also present some known results in Malliavin--Stein method that are used in this paper. For the first part, the reader can consult \cite{Nourdin-Peccati-Bible,nualart,DavidEulaliaBook} for further details. A comprehensive reference on the Malliavin--Stein method is the excellent monograph \cite{Nourdin-Peccati-Bible}.

\subsection{Elements of Gaussian Analysis}\label{ss:isonormal}

Let $ \HH$ be a real separable Hilbert space. For any $q\geq 1$, we write $ \HH^{\otimes q}$ and $ \HH^{\odot q}$ to
indicate, respectively, the $q$th tensor power and the $q$th symmetric tensor power of $ \HH$; we also set by convention
$ \HH^{\otimes 0} =  \HH^{\odot 0} =\R$. When $\HH = L^2(A,\mathcal{A}, \mu) =:L^2(\mu)$, where $\mu$ is a $\sigma$-finite
and non-atomic measure on the measurable space $(A,\mathcal{A})$, then $ \HH^{\otimes q} = L^2(A^q,\mathcal{A}^q,\mu^q)=:L^2(\mu^q)$, and $ \HH^{\odot q} = L_s^2(A^q,\mathcal{A}^q,\mu^q) := L_s^2(\mu^q)$,
where $L_s^2(\mu^q)$ stands for the subspace of $L^2(\mu^q)$ composed of those functions that are $\mu^q$-almost everywhere symmetric. We denote by $W=\{W(h) : h\in  \HH\}$
an {\it isonormal Gaussian process} over $ \HH$. This means that $W$ is a centered Gaussian family, defined on some probability space $(\Omega ,\mathcal{F},P)$, with a covariance structure given by the relation
$\E\left[ W(h)W(g)\right] =\langle h,g\rangle _{ \HH}$. We also assume that $\mathcal{F}=\sigma(W)$, that is, $\mathcal{F}$ is generated by $W$, and use the shorthand notation $L^2(\Omega) := L^2(\Omega, \mathcal{F}, P)$.

For every $q\geq 1$, the symbol $C_{q}$ stands for the $q$th {\it Wiener chaos} of $W$, defined as the closed linear subspace of $L^2(\Omega)$
generated by the family $\{H_{q}(W(h)) : h\in  \HH,\left\| h\right\| _{ \HH}=1\}$, where $H_{q}$ is the $q$th Hermite polynomial, defined as follows:
\begin{equation}\label{hq}
H_q(x) = (-1)^q e^{\frac{x^2}{2}}
\frac{d^q}{dx^q} \big( e^{-\frac{x^2}{2}} \big).
\end{equation}
We write by convention $C_{0} = \mathbb{R}$. For
any $q\geq 1$, the mapping $I_{q}(h^{\otimes q})=H_{q}(W(h))$ can be extended to a
linear isometry between the symmetric tensor product $ \HH^{\odot q}$
(equipped with the modified norm $\sqrt{q!}\left\| \cdot \right\| _{ \HH^{\otimes q}}$)
and the $q$th Wiener chaos $C_{q}$. For $q=0$, we write by convention $I_{0}(c)=c$, $c\in\mathbb{R}$.

It is well-known that $L^2(\Omega)$ can be decomposed into the infinite orthogonal sum of the spaces $C_{q}$: this means that any square-integrable random variable
$F\in L^2(\Omega)$ admits the following {\it Wiener-It\^{o} chaotic expansion}
\begin{equation}
F=\sum_{q=0}^{\infty }I_{q}(f_{q}),  \label{E}
\end{equation}
where the series converges in $L^2(\Omega)$, $f_{0}=E[F]$, and the kernels $f_{q}\in  \HH^{\odot q}$, $q\geq 1$, are
uniquely determined by $F$. For every $q\geq 0$, we denote by $J_{q}$ the
orthogonal projection operator on the $q$th Wiener chaos. In particular, if
$F\in L^2(\Omega)$ has the form (\ref{E}), then
$J_{q}F=I_{q}(f_{q})$ for every $q\geq 0$.

Let $\{e_{k},\,k\geq 1\}$ be a complete orthonormal system in $\HH$. Given $f\in  \HH^{\odot p}$ and $g\in \HH^{\odot q}$, for every
$r=0,\ldots ,p\wedge q$, the \textit{contraction} of $f$ and $g$ of order $r$
is the element of $ \HH^{\otimes (p+q-2r)}$ defined by
\begin{equation}
f\otimes _{r}g=\sum_{i_{1},\ldots ,i_{r}=1}^{\infty }\langle
f,e_{i_{1}}\otimes \ldots \otimes e_{i_{r}}\rangle _{ \HH^{\otimes
		r}}\otimes \langle g,e_{i_{1}}\otimes \ldots \otimes e_{i_{r}}
\rangle_{ \HH^{\otimes r}}.  \label{v2}
\end{equation}
Notice that the definition of $f\otimes_r g$ does not depend
on the particular choice of $\{e_k,\,k\geq 1\}$, and that
$f\otimes _{r}g$ is not necessarily symmetric; we denote its
symmetrization by $f\widetilde{\otimes }_{r}g\in  \HH^{\odot (p+q-2r)}$.
Moreover, $f\otimes _{0}g=f\otimes g$ equals the tensor product of $f$ and
$g$ while, for $p=q$, $f\otimes _{q}g=\langle f,g\rangle _{ \HH^{\otimes q}}$.
When $\HH = L^2(A,\mathcal{A},\mu)$ and $r=1,...,p\wedge q$, the contraction $f\otimes _{r}g$ is the element of $L^2(\mu^{p+q-2r})$ given by
\begin{eqnarray}\label{e:contraction}
 f\otimes _{r}g (x_1,...,x_{p+q-2r})  && = \int_{A^r} f(x_1,...,x_{p-r},a_1,...,a_r) \notag\\
&& \quad\quad\quad\quad \times g(x_{p-r+1},...,x_{p+q-2r},a_1,...,a_r)d\mu(a_1)...d\mu(a_r). \notag
\end{eqnarray}

It is a standard fact of Gaussian analysis that the following {\it multiplication formula} holds: if $f\in  \HH^{\odot p}$ and $g\in  \HH^{\odot q}$, then
\begin{eqnarray}\label{multiplication}
I_p(f) I_q(g) = \sum_{r=0}^{p \wedge q} r! {p \choose r}{ q \choose r} I_{p+q-2r} (f\widetilde{\otimes}_{r}g).
\end{eqnarray}
\smallskip

We now introduce some basic elements of the Malliavin calculus with respect
to the isonormal Gaussian process $W$. Let $\mathcal{S}$
be the set of all
cylindrical random variables of
the form
\begin{equation}
F=g\left( W(\phi _{1}),\ldots ,W(\phi _{n})\right) ,  \label{v3}
\end{equation}
where $n\geq 1$, $g:\mathbb{R}^{n}\rightarrow \mathbb{R}$ is an infinitely
differentiable function such that its partial derivatives have polynomial growth, and $\phi _{i}\in  \HH$,
$i=1,\ldots,n$.
The {\it Malliavin derivative}  of $F$ with respect to $W$ is the element of $L^2(\Omega , \HH)$ defined as
\begin{equation*}
DF\;=\;\sum_{i=1}^{n}\frac{\partial g}{\partial x_{i}}\left( W(\phi_{1}),\ldots ,W(\phi _{n})\right) \phi _{i}.
\end{equation*}
In particular, $DW(h)=h$ for every $h\in  \HH$. By iteration, one can define the $m$th derivative $D^{m}F$, which is an element of $L^2(\Omega , \HH^{\odot m})$,
for every $m\geq 2$.
For $m\geq 1$ and $p\geq 1$, ${\mathbb{D}}^{m,p}$ denotes the closure of
$\mathcal{S}$ with respect to the norm $\Vert \cdot \Vert _{m,p}$, defined by
the relation
\begin{equation*}
\Vert F\Vert _{m,p}^{p}\;=\;\E\left[ |F|^{p}\right] +\sum_{i=1}^{m}\E\left[
\Vert D^{i}F\Vert _{ \HH^{\otimes i}}^{p}\right].
\end{equation*}
We often use the (canonical) notation $\mathbb{D}^{\infty} := \bigcap_{m\geq 1}
\bigcap_{p\geq 1}\mathbb{D}^{m,p}$.


The Malliavin derivative $D$ obeys the following \textsl{chain rule}. If
$\varphi :\mathbb{R}^{n}\rightarrow \mathbb{R}$ is continuously
differentiable with bounded partial derivatives and if $F=(F_{1},\ldots
,F_{n})$ is a vector of elements of ${\mathbb{D}}^{1,2}$, then $\varphi
(F)\in {\mathbb{D}}^{1,2}$ and
\begin{equation}\label{e:chainrule}
D\,\varphi (F)=\sum_{i=1}^{n}\frac{\partial \varphi }{\partial x_{i}}(F)DF_{i}.
\end{equation}

Note also that a random variable $F$ as in (\ref{E}) is in ${\mathbb{D}}^{1,2}$ if and only if
$\sum_{q=1}^{\infty }q\|J_qF\|^2_{L^2(\Omega)}<\infty$
and in this case one has the following explicit relation: $$\E\left[ \Vert DF\Vert _{ \HH}^{2}\right]
=\sum_{q=1}^{\infty }q\|J_qF\|^2_{L^2(\Omega)}.$$ If $ \HH=
L^{2}(A,\mathcal{A},\mu )$ (with $\mu $ non-atomic), then the
derivative of a random variable $F$ as in (\ref{E}) can be identified with
the element of $L^2(A \times \Omega )$ given by
\begin{equation}
D_{t}F=\sum_{q=1}^{\infty }qI_{q-1}\left( f_{q}(\cdot ,t)\right) ,\quad t \in A.  \label{dtf}
\end{equation}


The operator $L$, defined as $L=-\sum_{q=0}^{\infty }qJ_{q}$, is the {\it infinitesimal generator of the Ornstein-Uhlenbeck semigroup}. The domain of $L$ is
\begin{equation*}
\mathrm{Dom}L=\{F\in L^2(\Omega ):\sum_{q=1}^{\infty }q^{2}\left\|
J_{q}F\right\| _{L^2(\Omega )}^{2}<\infty \}=\mathbb{D}^{2,2}\text{.}
\end{equation*}


For any $F \in L^2(\Omega )$, we define $L^{-1}F =-\sum_{q=1}^{\infty }\frac{1}{q} J_{q}(F)$. The operator $L^{-1}$ is called the
\textit{pseudo-inverse} of $L$. Indeed, for any $F \in L^2(\Omega )$, we have that $L^{-1} F \in  \mathrm{Dom}L
= \mathbb{D}^{2,2}$,
and
\begin{equation}\label{Lmoins1}
LL^{-1} F = F - \E(F).
\end{equation}



\subsection{Malliavin--Stein method:   selective results}\label{App:B-MS}
Next, we collect some known findings in the realm of Malliavin--Stein method that we have used in Section  \ref{sec:BM}. We begin with the celebrated {\it fourth moment theorem}.

\begin{thm}[Fourth Moment Theorem and Ramifications, see  \cite{n-p-05,Nualart-Latorre,NP-PTRF,n-p-optimal,n-p-AOP-Exact-Asymptotic}] \label{thm:4MT}
Fix $q \ge 2$. Let $F_n = I_q (f_n), n\ge 1$ be a sequence of elements belonging to the $q$th Wiener chaos of some isonormal Gaussian process $W = \{ W(h)  :  h \in \HH \}$ such that $\E[F^2_n]= q!  \Vert f_n \Vert^2_{\HH^{\otimes q}} =1$ for every $n \ge 1$.
\begin{description}
\item[(a)] Then the following asymptotic statements are equivalent as $n \to \infty$:
\begin{itemize}
	\item[(a)] $F_n$ converges in distribution towards $\mathcal{ N }(0,1)$.
	\item[(b)]  $\E[F^4_n]  \to 3$.
	\item[(c)]  $\Vert    f_n  \otimes_r f_n  \Vert_{\HH^{\otimes (2q-2r)}}  \to 0$ for $r=1,...,q-1$.
	\item[(d)]   $\Vert  DF_n \Vert^2_{\HH}   \to q  $  in $L^2$.
	\end{itemize}
\item[(b)] Furthermore, whenever one of the equivalent statements at item {\bf (a)} take place then there exist two constants $C_1$ and $C_2$ (independent of $n$) such that the following optimal rate of convergence in total variation distance holds:
\[ C_1 \,  \max\{ \abs{\kappa_3(F_n)}, \kappa_4(F_n) \} \leq d_{TV}(F_n,N) \leq  C_2 \, \max\{ \abs{\kappa_3(F_n)}, \kappa_4(F_n)  \}.\]
\item[(c)] Assume one of the  equivalent statements at item {\bf (a)} take place. Let $G_n$, $n\ge 1$ be a sequence  of the form
\[  G_n  =  \sum_{p=1}^{M}  I_p (g^{(p)}_n)  \]
for  $M\ge 1$  (independent of $n$) and some kernels $g^{(p)}_n  \in  \HH^{ \odot p},  p=1,...,M$.  Suppose that  as $n$ tends to infinity,
\[ \E[G^2_n]  = \sum_{p=1}^{M} p!  \Vert   g^{(p)}_n \Vert^2_{\HH^{\otimes p}}  \to  c^2  >0,  \quad   \Vert  g^{(p)}_n \otimes_r g^{(p)}_n \Vert_{\HH^{\otimes (2p-2r)}}  \to 0,   \quad \forall \, r=1,...,p-1  \]
and every $p=1,...,M$. If furthermore,  sequence $\E[F_nG_n] \to  \rho$, then sequence $(F_n,G_n)$ converges in distribution towards a two dimensional centered Gaussian vector $(N_1,N_2)$ with $\E[N^2_1]=1$, $\E[N^2_2] =c^2$, and $\E[N_1 N_2] =\rho$.
\end{description}
\end{thm}

\begin{thm}[Peccati-Tudor Multidimensional Fourth Moment Theorem \cite{Nourdin-Peccati-Bible},  Theorem 6.2.3] \label{thm:4MT-MultiDim}
Fix $d \ge 2$, and $q_1,...,q_d \ge 1$. Let $F_n  =  ( F_{1,n},...,F_{d,n}  ) = (  I_{q_1}(f_{1,n}),...,I_{q_d} (f_{d,n})  ),   n\ge 1$  with the kernels $f_{n,j} \in  \HH^{\odot j}$ for $j=1,...,d$ and every $n$.  Let $N \sim \mathcal{ N }_d (0,C)$ denote a $d$ dimensional centered Gaussian vector with a symmetric, non-negative covariance matrix $C$. Assume that $\E[F_{i,n} F_{n,j}]  \to C_{i,j}$ as $n \to \infty$. Then the following asymptotic statements are equivalent.
\begin{itemize}
	\item[(a)] $F_n  \to N$  is distribution.
	\item[(b)] for every $j=1,...,d$, sequence $F_{j,n} \to \mathcal{ N }(0,C_{j,j})$ in distribution.
	\end{itemize}
	
	\end{thm}
Now we recall Breuer-Major Theorem, see \cite{Breuer} or \cite[Theorem 7.2.4]{Nourdin-Peccati-Bible} for a modern treatment that is a cornerstone piece in Section \ref{sec:BM}.
\begin{thm}\label{thm:Breuer-Major}
	Let $X =\{X_k,  k\in\mathbb{Z}\}$ be a centered Gaussian stationary sequence with unit variance and set $r(k)=\mathbb{E}[X_{0}X_{k}]$ for every $k \in \mathbb{Z}$. Let $\gamma$ be the standard normal $\mathcal{N}(0,1)$ distribution and $f\in L^{2}(\mathbb{R}, \gamma)$ be a fixed deterministic function such that $\mathbb{E}[f(X_1)]=0$ and $f$ has Hermite rank $d\geq 1$, that means, that $f$ admits the Hermite expansion
	\begin{equation*}
	f(x)=\sum_{j=d}^{\infty}a_j {H_j}(x),
	\end{equation*}
	where $H_{j}$ is the $j$-Hermite polynomial,  and $a_d \not =0$. Define $V_{n}=\frac{1}{\sqrt{n}}\sum_{k=1}^{n}f(X_k)$.
	Suppose that $\sum_{\nu\in\mathbb{Z}}|r(\nu)|^{d}<\infty$. Then
	\begin{equation*}
	\sigma^2 : =\sum_{j=d}^{\infty}j! a^{2}_{j}\sum_{\nu\in\mathbb{Z}}r(\nu)^{j}\in [0,\infty),
	\end{equation*}
	and the convergence
	\begin{equation*}
	V_{n}\limd \mathcal{N}(0,\sigma^2)
	\end{equation*}
	holds as $n\to\infty$.
\end{thm}

\begin{thm}[See \cite{n-p-y19}]\label{thm:BM-TV-Rate}
	Let $N \sim \mathcal{ N }(0,1)$, and $X =\{X_k,  k\in\mathbb{Z}\}$ be a centered Gaussian stationary sequence with unit variance and covariance function $r(k)=\mathbb{E}[X_{0}X_{k}]$. Let $\gamma$ be the standard normal $\mathcal{N}(0,1)$ distribution and $f\in \mathbb{D}^{1,4} \subseteq L^2(\mathbb{R}, \gamma)$ be a fixed deterministic function such that $\mathbb{E}[f(X_1)]=0$. Let   $V_{n}=\frac{1}{\sqrt{n}}\sum_{k=1}^{n}f(X_k)$, and $\sigma^2_n =  \Var \left( V_n \right)$. Define  $F_n : =  \frac{V_n}{\sigma_n}$. Then there exists an explicit constant $C=C(f)$ such that for every $n \in \N$,
	\begin{equation}\label{eq:BM-TV-Rate}
	d_{TV}  (F_n,N)  \le   \frac{C(f)}{\sigma^2_n}  n^{- \frac{1}{2}}   \,  \left(  \sum_{\vert k \vert < n}  \vert  r (k) \vert   \right)^{\frac{3}{2}}
	\end{equation}
	\end{thm}

\begin{thm}[See \cite{HU2}] \label{thm:BM-Density-Rate}
Let $N \sim \mathcal{ N }(0,1)$. Assume that $X =\{X_k,  k\in\mathbb{Z}\}$ is a centered Gaussian stationary sequence with unit variance and covariance function $r(k)=\mathbb{E}[X_{0}X_{k}]$ whose spectral density function $f_r$ satisfies in   $ \log(f_r )\in L^1[-\pi,\pi]$. Fix $2 \le  d \le q$. Let $V_n =  \frac{1}{\sqrt{n}}  \sum_{k=1}^{n}  \sum_{j=d}^{q} a_j H_j (X_k)$ where $a_j \in \R$ for $d \le j \le q$, and that $\sigma^2_n : = \Var \left(   V_n \right)$. Define  $F_n :=  \frac{V_n}{\sigma_n}$.
\begin{description}
	\item[(a)] Assume further that $\sigma^2 :=  \sum_{j=d}^{q}  j! a^2_j  \sum_{\nu \in \Z}   r(\nu) ^j \in (0,\infty)$. Then, for every $m \ge 0 $ as  $n$ tends to infinity
	\begin{equation*}
\Big \Vert  p^{(m)}_n   - p^{(m)}_N   \Big \Vert_{L^{\infty}(\R)}:=
  \sup_{x \in \R}	  \Big  \vert     p^{(m)}_n (x)  -   p^{(m)}_N (x)  \Big \vert    \longrightarrow   0
	\end{equation*}
	where here $p^{(m)}_n$ and $p^{(m)}_N$ denote the $m$th derivative of density function of random variables $F_n$ and $N$ respectively.
	\item[(b)] In particular, if $q=d$ (in other words the sequence $F_n$ belongs to the fixed Wiener chaos of order $d$), then for all $m \ge 0$  there exist  $n_0 \in \N$  and a constant $C$ (depending only on $m$ and $q$) such that for all $n \ge n_0$  we have
	
	\begin{equation}\label{eq:Quantitative-Density-Convergence}
	\Big \Vert p^{(m)}_n - p^{(m)}_N  \Big \Vert_{L^{\infty}(\R)}  \le C \, \sqrt{        \E   \left[F^4_n \right] -     3       }.
	\end{equation}
	\end{description}
	\end{thm}

\begin{thm}[See   \cite{n-p-AOP-Exact-Asymptotic}]  \label{thm:BM-Exact-Asymptotic}
Let $(F_n   :   n\ge 1)$  be a sequence of centered square integrable functionals of some isonromal Gaussian process $W  =\{  W(h)   :  h \in \HH \}$ such that  $\E  [F^2_n]  \to 1$  as $n $ tends to infinity.  Assume further the following  assumptions hold:
\begin{itemize}   	
	\item[(a)] for every $n$,  the random variable $F_n \in \mathbb{D}^{1,2}$, and that the law of $F_n$ is absolutely continuous with respect to the Lebesgue measure.
	\item[(b)] the quantity $\varphi(n):=  \sqrt{  \E  \left[  (1 -  \langle   DF_n  , -DL^{-1}F_n  \rangle_{\HH}  )^2  \right]  }$ is such that:   (i)  $\varphi(n) < \infty$ for every $n$, (ii)  $\varphi(n) \to 0$ as $n $ tends to infinity, and (iii)  there exists $m \in \N$ such that $\varphi(n) > 0$ for all $n \ge m$.
	\item[(c)] as $n$ tends to infinity,
	 \[   \left( F_n  ,\frac{1 -  \langle   DF_n  , -DL^{-1}F_n  \rangle_{\HH} }{\varphi(n)}  \right)   \stackrel{d}{\longrightarrow}  (N_1,N_2) \]
	 where $(N_1,N_2)$ is a two dimensional centered Gaussian vector with $\E[N^2_1]= \E[N^2_2]=1$, and $\E[N_1 N_2] = \rho$.
	\end{itemize}
Then, we have $d_{Kol} (F_n, N)   \le   \varphi(n)$, and moreover, for every $z \in  \R$ as $n \to \infty$:
\[  \frac{ \mathbb{P}  \left(  F_n
	\le z \right)  -   \mathbb{P} (N \le z)}{\varphi(n)}   \longrightarrow  \frac{\rho}{3 \sqrt{2\pi}} (z^2 -1) e ^{- \frac{z^2}{2}}. \]
	\end{thm}

\begin{thm}[See \cite{b-n-t-almost-sure}] \label{thm:BM-AlmostSure}
	Let $N \sim \mathcal{ N }(0,1)$. Assume that $X =\{X_k,  k\in\mathbb{Z}\}$ is a centered Gaussian stationary sequence with unit variance and covariance function $r(k)=\mathbb{E}[X_{0}X_{k}]$ such that $\sum_{\nu \in \Z}  \vert  r(\nu) \vert  <  \infty$. Assume that $f \in L^2(\R,\gamma)$ is a non-constant function of the class $C^2(\R)$ so that $\E[f''(N)^4]< \infty$ and that $\E_\gamma[f]=0$.  Let $V_n =  \frac{1}{\sqrt{n}}  \sum_{k=1}^{n} f (X_k)$, and that $\sigma^2_n : = \Var \left(   V_n \right)$. Define  $F_n :=  \frac{V_n}{\sigma_n}$. If as $n$ tends to infinity, $\sigma^2_n  \to \sigma^2 > 0$,  then the sequence $(F_n : n \ge 1)$  converges in distribution towards $N$, and moreover it satisfies an ASCLT meaning that, almost surely, for every bounded continuous function $\varphi : \R \to \R$ it holds that
	\[  \frac{1}{\log n} \sum_{k=1}^{n} \frac{1}{k} \varphi (F_n)  \longrightarrow \E \left[  \varphi(N) \right], \quad \text{ as } \quad n \to \infty. \]
\end{thm}

\bibliographystyle{abbrv}
\bibliography{FBMVSTFBM_FINAL_}

\end{document}